\documentclass[11pt,reqno]{amsart}
\usepackage[colorlinks=true,citecolor=green]{hyperref}
\usepackage{mathptmx, amsmath, amssymb, amsfonts, amsthm, mathptmx, enumerate, color,mathrsfs}
\usepackage{graphicx}
\usepackage{epstopdf}
\begin{document}
	\title[\hfilneg \hfil On the concentration-compactness principle]
	{On the concentration-compactness principle for anisotropic variable exponent Sobolev spaces and its applications}	
	\author[N. Chems Eddine, M.A. Ragusa, D.D. Repov\v{s}
	\hfil \hfilneg]
	{Nabil CHEMS EDDINE, Maria Alessandra RAGUSA$^*$, Du\v{s}an D. REPOV\v{S}}  		
		\address{Nabil Chems Eddine\newline
		Laboratory of Mathematical Analysis and Applications,  
		Department of Mathematics, 
		Faculty of Sciences, 
		Mohammed V University, 
		Rabat, Morocco.}
	\email{nabil.chemseddine@um5r.ac.ma}		
		\address{ Maria Alessandra Ragusa \newline
		Dipartimento di Matematica e Informatica, NANOMED, Research Centre for Nanomedicine and Pharmaceutical Nanotechnology, Universitá di Catania, Catania, Italy.\newline
		Faculty of Fundamental Science, Industrial University, Ho Chi Minh City, Vietnam.}
	\email{maragusa@dmi.unict.it}	
		\address{  Du\v{s}an D. Repov\v{s} \newline
	  Faculty of Education, University of Ljubljana,  Ljubljana, Slovenia.\newline
		Faculty of Mathematics and Physics, University of Ljubljana,  Ljubljana, Slovenia.\newline
		 Institute of Mathematics, Physics and Mechanics,  Ljubljana, Slovenia.}
	\email{dusan.repovs@guest.arnes.si}	
	\thanks{$^*$Corresponding author}
	\begin{abstract}
		 We obtain critical embeddings and the concentration-compactness principle for the anisotropic variable exponent Sobolev spaces. As an application of these results,we confirm the existence of and find infinitely many nontrivial solutions for a class of nonlinear critical anisotropic elliptic equations involving variable exponents and two real parameters. With the groundwork laid in this work, there is potential for future extensions, particularly in extending the concentration-compactness principle to anisotropic fractional order Sobolev spaces with variable exponents in bounded domains. This extension could find applications in solving the generalized fractional Brezis-Nirenberg problem.
	\end{abstract}	
	\maketitle
	\numberwithin{equation}{section}
	\newtheorem{theorem}{Theorem} 
	\newtheorem{lemma}{Lemma}
		\newtheorem{definition}{Definition}
	\newtheorem{proposition}{Proposition}
	\newtheorem{remark}{Remark}
		\newtheorem{example}{Example}
		\newtheorem{question}{Problem}
	\allowdisplaybreaks	
 
	\section{Introduction}	
In recent years, increasing attention has been paid to the study of differential and partial differential equations involving the variable exponent in general, and anisotropic equations with different orders of derivation in different directions in particular. The main interest in studying such problems has been stimulated by their various applications in physical and related sciences. Indeed, there are many applications concerning nonlinear elasticity problems,  contact mechanics, electrorheological fluids, robotics, space technology, image processing,  flow in porous media, etc. (for more details see  Antontsev et al. \cite{AnDi},  Antontsev and Rodrigues \cite{AnRo}, Bear \cite{Bear}, Boureanu et al. \cite{BoureanuMatei}, Chen et al. \cite{ChenL}, Diening \cite{Diening}, R\u{a}dulescu and Repov\v{s} \cite{radulescu1}, R\.{u}\v{z}i\v{c}ka \cite{R}, Simmonds \cite{SobSimmonds}, Stanway et al. \cite{Sob5}, Zhikov \cite{Zhikov}, and the references therein).\par	

To the best of our knowledge, anisotropic equations with different orders of derivation in different directions, involving critical variable exponents have never been studied before. In the subcritical case, we refer the reader to the papers Boureanu and Udrea \cite{Boureanu1}, Boureanu and R\u{a}dulescu \cite{Boureanu}, Fan \cite{fananiso}, Ji \cite{ji},  Mih\u{a}ilescu et al. \cite{m1,m2}, and
Ourraoui and Ragusa \cite{Ourraoui}.\par	

One of the main points in the study of these equations is the generalization of the well-known Anisotropic Sobolev Immersion Theorem:
If $\Omega$ is a subset of $\mathbb{R}^{N}$ and $\overrightarrow{p}:\Omega\to\mathbb{R}^N$ is the vector function $ \overrightarrow{p}(x)=\left(p_1(x), \dots,p_N(x)\right) $
such that  $1< \displaystyle p_i^-:=\inf_{x\in \Omega}p_i(x)\leq p_i^+:=\sup_{x\in \Omega}p_i(x)<N,$  for all $i\in\{1,\dots,N\}$, then there is a
continuous
embedding
(resp. a compact embedding
$W^{1,\overrightarrow{p}(x)}(\Omega)\hookrightarrow L^{h(x)}(\Omega)),$
if the exponent $h: \Omega\to [1,+\infty)$  satisfies
$
h(x)\leq P^{\ast}(x)$ for continuous embedding
(resp. $h(x)< P^{\ast}(x),$ for compact embedding),
where  $P^{\ast}(x)$ is  the critical Sobolev exponent.

We mention the most important results on this topics. When $\Omega \subset \mathbb{R}^N(N\geq3)$ is a bounded domain with smooth boundary, Mih\u{a}ilescu et al. \cite{m1}
proved that for all  continuous function $h$ satisfying
$
1< h(x) < P^{\ast}(x):= \max\{ p_{\max}^-, p_m^{\ast}\},$
where
$
p_{\max}^{-}:=\max_{1\leq i\leq N}\{p_i^-\},
$
$
p_m^{\ast}:=N / (\sum_{i=1}^{N} \frac{1}{p_i^{-}}-1),
$
$
\sum_{i=1}^{N} \frac{1}{p_i^{-}}>1,
$
$W_0^{1,\overrightarrow{p}(x)}(\Omega)$  is
compactly embeddable in
$L^{h(x)}(\Omega)$. Subsequently,
Ji \cite{ji}
showed also that for all  continuous function $h$ satisfying
$
1< h(x) < P^{\ast}(x):=\frac{Np_{m}^{-}}{N-p_{m}^{-}}, \text{ where  } 	p_{m}^{-}:= \min_{1\leq i\leq N}\{p_i^-\},
$
the space
$W^{1,\overrightarrow{p}(x)}(\Omega)$  is  compactly embeddable
in
$L^{h(x)}(\Omega)$. \par
Note that the cited results, in which critical exponent  $\max\{ p_{\max}^-, p_m^{\ast}\}$ or $\frac{Np_{m}^{-}}{N-p_{m}^{-}}$ is constant
exponent, are optimal in the environment of constant exponent Lebesgue spaces. In the present work, we shall  be interested
in extending these results, by giving sufficient conditions for
$p_i, \  i=1,2,\dots,N$
so that $W^{1,\overrightarrow{p}(x)}(\Omega)$
is embeddable in $L^{P^{\ast}(x)}(\Omega),$ where $\Omega \subset \mathbb{R}^N(N\geq2)$ is a bounded domain with smooth boundary and $P^{\ast}(x)=\frac{Np_{m}(x)}{N-p_{m}(x)}$
is optimal in the environment of variable exponent Lebesgue spaces. We conclude this paragraph with the following open problem.
\begin{question}
	What are sufficient conditions for $p_i, i=1,2,\dots,N$ and $\Omega \subset \mathbb{R}^N$, so that $W^{1,\overrightarrow{p}(x)}(\Omega)$ or $W_0^{1,\overrightarrow{p}(x)}(\Omega)$ is embeddable in
	$L^{P^{\ast}(x)}(\Omega),$  where
	$
	P^{\ast}(x):= \frac{N\bar{p}(x)}{N-\bar{p}(x)} \text{ and } \bar{p}(x)
	:=\frac{N}{\sum_{i=1}^{N} \frac{1}{p_i(x)}}.
	$
\end{question}\par 
On the other hand, when $p_i(x)= p(x)$ or $p$ is constant for all $i\in \{1,2,\dots,N\}$, the class of elliptic equations involving critical growth has received great attention following the seminal work of Brezis and Nirenberg \cite{Brezis} in 1983 for Laplacian equations. Since then, there have been extensions of \cite{Brezis} in many directions, see e.g.,  Servadei and Valdinoci \cite{Servadei-Valdinoci1,Servadei-Valdinoci2}.

The principal challenge in solving elliptic problems characterized by critical growth lies in the absence of compactness when embedding Sobolev spaces into Lebesgue spaces within the framework of variational methods. To overcome this obstacle, Lions \cite{Lions1} introduced the concentration-compactness principle (CCP) in 1985 to establish the precompactness of minimizing sequences or Palais-Smale (PS) sequences. For bounded domains, a variable exponent adaptation of Lions' concentration-compactness principle was independently derived by Bonder and Silva \cite{Bonder}, as well as by Fu \cite{Fu}. Since then, numerous researchers have applied these findings to tackle critical elliptic problems involving variable exponents, see e.g., Alves and Ferreira \cite{Alves1}, Alves and Barreiro \cite{Alves0}, Chems Eddine and Ragusa \cite{Chems0,Chems1}, Fu and Zhang \cite{Fu1}, Ho and Sim \cite{HoSim}, Hurtado et al. \cite{Hurtado}, and the references therein.

For the fractional $p(x)$-Laplacian on bounded domains, the CCP was established in the linear case $p = 2$ by Palatucci and Pisante \cite{Palatucci-Pisante} and for constant $p$ by Mosconi and Squassina \cite{Mosconi-Squassina}. The CCP for the fractional
Sobolev spaces with variable exponents was extended by Ho and Kim \cite{Ho}. Using the concentration-compactness principle, they provide sufficient conditions for the existence of a nontrivial solution to the generalized fractional Brezis–Nirenberg problem. Notably, El Hamidi and Rakotoson \cite{El Hamidi1} extended the concentration-compactness principle to anisotropic Sobolev spaces with constant exponents when all $p_i$ are constant functions. This extension paved the way for demonstrating the attainment of a critical best Sobolev constant. Subsequently, various authors have effectively addressed critical problems involving the $\overrightarrow{p}$-Laplacian operator, as exemplified by the works of Alves and El Hamidi \cite{Alves} and Figueiredo et al. \cite{Figueiredo,Figueiredo1}.

Recently, Chaker et al. \cite{Chaker-Kim-Weidner} extended the concentration-compactness principle for the anisotropic fractional $\overrightarrow{p}$-Laplacian of mixed order to unbounded domains. Combining ideas from the anisotropic case as presented in this paper, from the nonlocal case as in Ho and Kim \cite{Ho}, and from the anisotropic nonlocal operator in \cite{Chaker-Kim-Weidner}, we anticipate that this combined approach will allow us to extend the concentration-compactness principle to anisotropic fractional order Sobolev spaces with variable exponents in bounded domains in the future. Such an extension holds the potential for applications in solving the anisotropic generalized fractional Brezis-Nirenberg problem.\par 

As mentioned previously, there were no prior results available for nonlinear anisotropic elliptic equations with variable critical growth until this article. However, it is worth noting that a subsequent paper by Chems Eddine et al. \cite{ChemsNguyenRagusa}, published after this article, utilized the results obtained in this work. They applied these findings to a specific class of critical anisotropic elliptic equations of Schrödinger-Kirchhoff-type. Although the crucial Sobolev immersion theorem holds for anisotropic Sobolev with variable exponents, we do not know if there are results for the critical Sobolev type embedding for the anisotropic variable exponent Sobolev spaces defined on a bounded domain, see e.g., Ji \cite{ji}, Fan \cite{fananiso}, Mih\u{a}ilescu et al. \cite{m1,m2}, and R\u{a}dulescu and Repov\v{s} \cite{radulescu1}. Because of this, our first aim is to obtain a critical embedding from anisotropic variable exponent Sobolev spaces into variable exponent Lebesgue spaces. We give sufficient conditions on the variable exponents, such as the log-Hölder type continuity condition on the minimum function of the exponents, to obtain such critical embedding (see Theorem \ref{theo23}). Using this critical embedding, we establish the extension of the Lions concentration-compactness principle for anisotropic variable exponent Sobolev spaces, inspired by Bonder and Silva \cite{Bonder}, Fu \cite{Fu}, Ho and Kim \cite{Ho},
and
Lions \cite{Lions1}, which are our second aim (Theorem \ref{ccp}). As an application of these results, we establish the existence and multiplicity of nontrivial solutions for the following class of nonhomogeneous anisotropic eigenvalue critical problems
\begin{equation}\label{e1.1}
	\begin{cases}
		\displaystyle- \sum_{i=1}^N\partial_{x_i}a_i (x,\partial_{x_i}u ) + \lambda | u|^{r(x)-2}u=| u|^{h(x)-2}u + \beta f(x,u) \  \text { in }\Omega, \\
		u=0 \   \text{on }\partial\Omega,
	\end{cases}
\end{equation}
where $\Omega \subset \mathbb{R}^N (N\geq 2)$ is a bounded domain with a Lipschitz boundary $\partial \Omega$,  $\lambda$ and $\beta$ are real parameters such that  $\beta$ is positive, functions $r, h$ and $p_i, i=1,2,\dots,N$ are continuous on $\overline{\Omega}$ and satisfy some conditions to be specified below, and $f: \Omega \times \mathbb{R}\to \mathbb{R}$ is a Carathéodory function with the potential
$F(x,\xi)=\int_{0}^{\xi}f(x,t)dt,$
that satisfies some conditions which will be specified later.
The differential operator $\displaystyle\sum_{i=1}^N\partial_{x_i}a_i (x,\partial_{x_i}u )$ was originally introduced by Boureanu and Rădulescu \cite{Boureanu}. This operator is a more general type of Laplacian operator.  The functions
$a_i(x,\xi)$  represent the continuous derivatives with respect to $\xi$ of the mapping  $A_i:\Omega\times\mathbb{R}\to\mathbb{R}$, denoted as $A_i=A_i(x,\xi)$, that is, $a_i(x,\xi)=\frac{\partial}{\partial \xi}A_i(x,\xi)$.\par

In this paper, we shall assume that the following hypotheses hold for all $1\leq i \leq N$ :
\begin{itemize}	
	\item[$(\textbf{\textit{A}}_1)$] There exist  $c_{a_i}>0$ such
	that
	$
	|a_i(x,\xi)|\leq c_{a_i}\left(g_i(x)+|\xi|^{p_i(x)-1}\right),
	\
	\hbox{for  a.e.}
	\
	x\in\Omega
	\
	\hbox{ and  all}
	\
	\xi\in\mathbb{R},
	$
	where the nonnegative functions
	$g_i$ belong to  $L^{p'_i(x)}(\Omega)$, with $\frac{1}{p_i(x)}+\frac{1}{p'_i(x)}=1$.		
	\item[$(\textbf{\textit{A}}_2)$] There exist positive constants $k_i$ such that
	$$
	k_i|\xi|^{p_i(x)}\leq a_i(x,\xi)\xi \leq p_i(x)\;A_i(x,\xi),
	\		
	\hbox{for  a.e.}
	\
	x\in\Omega
	\
	\hbox{ and  all}
	\
	\xi\in\mathbb{R}.
	$$
	\item[$(\textbf{\textit{A}}_3)$]  The functions $a_i$ satisfy
	$
	\big(a_i(x,\xi)-a_i(x,\eta)\big)(\xi-\eta)> 0,
	\hbox{ for  a.e.}
	\\
	x\in\Omega
	\
	\hbox{ and  all}
	\
	\xi,\eta\in\mathbb{R},
	\
	\xi\neq \eta.
	$
\end{itemize}

The main feature of  this paper is  establishing the existence and multiplicity of nontrivial solutions to problem \eqref{e1.1} under the critical growth condition $ \{x\in\Omega \colon h(x)=p_m^*(x)\}\neq\emptyset$, where $p_m^*(x):=\frac{Np_m(x)}{N-p_m(x)}$ with $p_m(x)
=	\displaystyle\min_{1\leq i\leq N}\{p_i(x)\}$  (see Theorem \ref{thm1} for existence and Theorem \ref{thm2} for multiplicity).
The main results of this paper are as follows (conditions $(\textbf{\textit{f}}_i)$
and
$(\textbf{\textit{H}})$
will be defined in Section \ref{sec 4}).
\begin{theorem} \label{thm1}
	Suppose that assumptions {\rm ($\textbf{\textit{A}}_1$)}-{\rm ($\textbf{\textit{A}}_3$)}, {\rm ($\textbf{\textit{f}}_1$)-($\textbf{\textit{f}}_3$)}, and {\rm ($\textbf{\textit{H}}$)} hold.
	Then  for all  $\lambda\in\mathbb{R}$ and $\beta> 0$, problem \eqref{e1.1}
	has at least one nontrivial weak solution.
\end{theorem}\par 
The second theorem concerns the case  $a_i(x,\xi):=|\xi|^{p_i(x)-2}\xi$, for all $i\in\{1,\dots,N\}$.
\begin{theorem} \label{thm2}
 Suppose that assumptions  {\rm ($\textbf{\textit{A}}_1$)}-{\rm ($\textbf{\textit{A}}_3$)}, {\rm ($\textbf{\textit{f}}_1$)-($\textbf{\textit{f}}_4$)}, and {\rm ($\textbf{\textit{H}}$)} hold.
	Then  for all  $\lambda\in\mathbb{R}$ and $\beta> 0$, problem \eqref{e1.1}
	has infinitely many weak solutions.
\end{theorem}

We give some examples, interesting from the mathematical point of view and with a wide range of applications in physics and other fields, that fall within the general class of equations which we shall  study in this paper, with adequate assumptions on functions $a_{i}$.
\begin{example} \label{example11}
	Let $a_i(x,\xi):=|\xi|^{p_i(x)-2}\xi.$
	Then  $A_i(x,\xi)=\frac{1}{p_i(x)}|\xi|^{p_i(x)}$ and  $a_i$ satisfy the assumptions $(\textbf{\textit{A}}_1),(\textbf{\textit{A}}_2)$ and $(\textbf{\textit{A}}_3)$ for all $i\in\{1,\dots,N\}$. Hence  equation \eqref{e1.1} becomes	
	\begin{equation}\label{example1.1}
		\begin{cases}
			\displaystyle	-\Delta_{\overrightarrow{p}(x)}(u)  + \lambda | u|^{r(x)-2}u=| u|^{h(x)-2}u + \beta f(x,u) \  \text { in }\Omega, \\
			u=0 \   \text{on }\partial\Omega.
		\end{cases}
	\end{equation}	
	The operator $ \Delta_{\overrightarrow{p}(x)}(u):= \sum_{i=1}^{N}\partial_{ x_{i}}\Big(\big|\partial_{ x_{i}}u\big|^{p_{i}(x)-2} \partial_{ x_{i}}u\Big)$
	is the so-called the $\overrightarrow{p}(x)$-Laplacian operator,
	when $p_i(x)=p(x)$ for all  $i=1,2,\dots,N.$
	The operator $\Delta_{\overrightarrow{p}(x)}u$ is the $p(x)$-Laplacian operator, i.e., $\Delta_{p(x)}u = \mathrm{div}(|\nabla u|^{p(x)-2}\nabla u)$, which coincides with the standard $p$-Laplacian when $p(x)=p$, and with the Laplacian when $p(x)=2$.
\end{example}
\begin{example} \label{example12}
	Let $a_i(x,\xi):=(1+|\xi|^2)^{\frac{p_i(x)-2}{2}}\xi.$
	Then  $A_i(x,\xi)=\frac{1}{p_i(x)}\big( (1+|\xi|^2)^{\frac{p_i(x)}{2}} -1\big)$ and  $a_i$ satisfy the assumptions $(\textbf{\textit{A}}_1),(\textbf{\textit{A}}_2)$ and $(\textbf{\textit{A}}_3)$ for all $i\in\{1,\dots,N\}$. Hence equation \eqref{e1.1} becomes	
	\begin{equation}\label{example1.2}
		\begin{cases}
			\displaystyle	-\sum_{i=1}^N\partial_{x_i}\bigg(\big(1+|\partial_{x_i}u
			|^2\big)^{(p_i(x)-2)/2}\partial_{x_i}u \bigg)  + \lambda | u|^{r(x)-2}u
			\\\qquad =| u|^{h(x)-2}u + \beta f(x,u) \text { in }\Omega, \\
			u=0 \   \text{on }\partial\Omega.
		\end{cases}
	\end{equation}		
	The operator $\sum_{i=1}^N\partial_{x_i}\bigg(\big(1+|\partial_{x_i}u
	|^2\big)^{(p_i(x)-2)/2}\partial_{x_i}u \bigg)$ is
	the
	so-called anisotropic variable mean curvature operator.
\end{example}

The paper is structured as follows: In Section \ref{sec 2} we give some preliminary properties of the variable exponent spaces. In Section \ref{sec 3} we establish an extension of the Lions concentration-compactness principle for anisotropic variable exponent Sobolev spaces. In Section \ref{sec 4} we study a class of nonlinear anisotropic elliptic equations with critical growth and establish the existence and multiplicity of solutions. In Section \ref{sec 5} we prove the main results (Theorems \ref{thm1} and  \ref{thm2}).		

\section{Functional framework} \label{sec 2}
In this section, we establish the notation and compile essential foundational results concerning variable exponent function spaces. These results will be recurrently employed in subsequent sections of the paper. \par
Throughout this paper, we assume that $\Omega$ is a bounded Lipschitz domain in $\mathbb{R}^N$ $(N\geq 2)$. We introduce the set $C_{+}(\overline{\Omega})$,  defined as
$$
C_{+}(\overline{\Omega})= \{ p: p \in C(\overline{\Omega}) , p(x)> 1 \
\text{ for a.e. } x \in \overline{\Omega} \}.
$$
 We denote by $C_{+}^{\text{log}}(\overline{\Omega})$  the set of functions $p\in C_{+}(\overline{\Omega})$ that satisfy the log-Holder continuity condition
$$\sup \left\{ |p(x)-p(y)| \log \frac{1}{|x-y|} : x,y\in \overline{\Omega}, 0< |x-y|<\frac{1}{2} \right\} < \infty.$$
For any $ p \in C_{+}(\overline{\Omega}),$ we define $p^{+}= \sup_{x\in \Omega}p(x)$ and $p^{-}= \inf_{x\in \Omega}p(x)$. We also introduce the variable exponent Lebesgue space as
$$L^{p(x)}(\Omega)= \{ u : u \text{ is a measurable real-valued function and }  \rho_{p}(u)  <\infty \},$$
where the functional $\rho_{p}:L^{p(x)}(\Omega) \to \mathbb{R}$ is defined as $\rho_{p}(u):= \int_{\Omega}|u(x)|^{p(x)}dx$.  We endow the space $L_{}^{p(x)}(\Omega)$ with the Luxemburg norm
$$
\|u\|_{L_{ }^{p(x)}(\Omega)}: = \inf \left\{ \tau >0: \rho_{p}\bigg(\frac{\left|u(x)\right|}{\tau}\bigg) \leq 1 \right \}.
$$
This norm results in $(L^{p(x)}(\Omega),\|u\|_{L^{p(x)}(\Omega)} )$ being a separable and reflexive Banach space (see, e.g., Kov\'a\v cik and R\'akosn\'{\i}k \cite[Theorem 2.5, Corollary 2.7]{KR}).  Let us now revisit some fundamental properties associated with Lebesgue spaces.
\begin{proposition}[Kov\'a\v cik and R\'akosn\'{\i}k {\cite[Theorem 2.8]{KR}}] \label{prop21}
	Consider variable exponents $q$ and $h$ belonging to the class $C_{+}(\overline{\Omega})$ and satisfying the condition $q\leq h$ within the domain $\Omega$. Under these conditions, the embedding $L^{h(x)}(\Omega )\hookrightarrow L^{q(x)}(\Omega )$ is continuous.
\end{proposition}\par 
Furthermore, the following H\"older-type inequality holds for all   $u\in
L^{p(x)}({\Omega})$ and $v\in L^{p'(x)}({\Omega})$
\begin{equation}\label{Hol}
		\begin{aligned}
	\Big| \int_{\Omega}u(x)v(x)dx\Big|
	&\leq \Big(\frac{1}{p^{-}}+ \frac{1}{(p')^{-}}\Big)\| u\|_{L^{p(x)}(\Omega)}\big\| v\big\| _{L^{p'(x)}(\Omega)}\\
	&\leq 2\big\| u\big\|_{L^{p(x)}(\Omega)}\big\| v\big\| _{L^{p'(x)}(\Omega)},
		\end{aligned}
\end{equation}
where $L^{p'(x)}({\Omega})$ is  the conjugate space (or the topological dual space) of $L^{p(x)}({\Omega})$, obtained by conjugating the exponent pointwise, that is,  $\frac{1}{p(x)}+\frac{1}{p'(x)}=1$ (see, e.g., Kov\'a\v cik and R\'akosn\'{\i}k \cite[Theorem 2.1, Corollary 2.7]{KR}).
Moreover, if  $u\in L^{p(x)}({\Omega})$ and $p<\infty$, then  we have the following properties (see for example Fan and Zhao \cite[Theorem 1.3, Theorem 1.4]{fanz}):
\begin{equation}\label{L23}
	\|u\|_{L^{p(x)}({\Omega})}<1\;(=1;\,>1)\; \    \text{ if and only if } \   \; \rho_p (u) <1\;(=1;\,>1),
\end{equation}
\begin{equation}\label{L24}
	\text{ if }\   	\|u\|_{L^{p(x)}({\Omega})}>1\; \   \text{ then } \   \;
	\|u\|_{L^{p(x)}({\Omega})}^{p^-}\leq \rho_p (u)
	\leq\|u\|_{L^{p(x)}({\Omega})}^{p^+},
\end{equation}
\begin{equation}\label{L25}
	\text{ if }\   	\|u\|_{L^{p(x)}({\Omega})}<1\; \   \text{ then } \    \;
	\|u\|_{L^{p(x)}({\Omega})}^{p^+}\leq \rho_p (u)
	\leq\|u\|_{L^{p(x)}({\Omega})}^{p^-}.
\end{equation}
As a result, we get
\begin{equation}\label{L225}
	\| u\|_{L^{p(x)}(\Omega)}^{p^{-}}-1
	\leq \rho_p (u)\leq \| u\|_{L^{p(x)}(\Omega)}^{p^{+}}+1, \   \text{ for all }  u \in L^{p(x)}({\Omega}).
\end{equation}
This leads to an important result that norm convergence and modular convergence are equivalent.
\begin{equation}\label{L26}
	\|u\|_{L^{p(x)}({\Omega})}\to 0\;(\to\infty)\; \   \text{ if and only if } \   \; \rho_p (u) \to 0\; (\to\infty).
\end{equation}
\begin{remark}\label{remark 1}
	The above properties of the modular and  norm hold for all \\ $L_{\mu}^{p(x)}(\Omega):=\{u: u \text{ is }  \mu-\text{measurable real-valued function and} \\  \displaystyle \int_{\Omega}|u(x)|^{p(x)}d\mu<\infty\},$
	where $\Omega \subset \mathbb{R}^N (N\geq 2)$ is a bounded open subset, $\mu$ is a measure on $\Omega$, and $p\in C_{+}(\overline{\Omega})$ (for more details see, e.g., Diening et al. {\cite[Chapter 3]{Dien}}).
\end{remark}\par

Now, let us turn our attention to the (isotropic) Sobolev space with a variable exponent. This space, denoted as \(W^{1,p(x)}(\Omega)\), consists of functions \(u\) belonging to \(L^{p(x)}(\Omega)\) whose partial derivatives \(\partial_{x_i}u\), for \(i \in \{1,\dots,N\}\), are also in \(L^{p(x)}(\Omega)\) in the weak sense. The norm for this space is given by
\[
\|u\|_{1,p(x)} := \|u\|_{W^{1,p(x)}(\Omega)} := \|u\|_{L^{p(x)}(\Omega)} + \|\nabla u\|_{L^{p(x)}(\Omega)}   \text{ for all } u\in W^{1,p(x)}(\Omega ),
\]
where $\nabla u$ represents the gradient of $u$.  The Sobolev space with zero boundary values, denoted as $W_0^{1,p(x)}(\Omega )$, is defined as the closure of the set of smooth functions with compact support,  $C_0^{\infty }(\Omega )$, within $W^{1,p(x)}(\Omega )$. Its norm is given by
\[
\| u\|_{p(x)} = \| \nabla u\|_{L^{p(x)}(\Omega)} \text{ for all } u\in W_0^{1,p(x)}(\Omega ).
\]
It is worth noting that both $W^{1,p(x)}(\Omega)$ and $W_0^{1,p(x)}(\Omega)$ are separable and reflexive Banach spaces, as established in Kov\'a\v cik and R\'akosn\'{\i}k \cite[Theorem 3.1]{KR}.
Additionally, we introduce the concept of the critical Sobolev exponent, denoted as \(p^{\ast }(x)\), which is defined as follows
\[
p^{\ast }(x) = \begin{cases}
	\frac{Np(x)}{N-p(x)} &\text{if } p(x)<N \\
	+\infty &\text{if } p(x)\geq N.
\end{cases}
\]
Now, let us highlight the crucial embeddings of the space \(W^{1,p(x)}(\Omega)\).
\begin{proposition}[Diening et al. {\cite[Theorem 8.4.2.]{Dien}}, Edmunds and Rakosnik {\cite[Theorem 1.1]{Edmu}}] \label{prop24}
	Consider $p\in C_{+}^{\log}(\overline{\Omega})$ satisfying $p^+<N$, and $h\in C(\overline{\Omega})$ with $1\leq h(x) \leq p^{\ast}(x)$ for all $x\in \overline{\Omega}$. Under these conditions, we have a continuous embedding $W^{1,p(x)}(\Omega)\hookrightarrow L^{h(x)}(\Omega)$. Furthermore, if we additionally assume $1\leq h(x)< p^{\ast}(x)$ for all $x\in \overline{\Omega}$, then this embedding is also compact.
\end{proposition}\par 

For a comprehensive exploration of the properties of Lebesgue-Sobolev spaces with variable exponents, we recommend that the reader consult the works  of Cruz-Uribe and Fiorenza \cite{Cruz-UribeFiorenza}, Diening et al. \cite{Dien} and Kov\'a\v cik and R\'akosn\'{\i}k \cite{KR}.\par
Now, we expand our discussion to include the anisotropic Sobolev space denoted as $W^{1,\overrightarrow{p}(x)}(\Omega)$, where $\overrightarrow{p}:\overline\Omega\to\mathbb{R}^N$ is a vector function defined as $
\overrightarrow{p}(x)=\left(p_1(x), \dots,p_N(x)\right),
$
with each component $p_i \in C_+(\overline\Omega)$ satisfying $1< p_i^-\leq p_i^+<N<\infty$ for all $i\in\{1,\dots,N\}$. Additionally, we define
$
p_m(x)=\min\{p_1(x),\dots,p_N(x)\},\   p_M(x)=\max\{p_1(x),\dots,p_N(x)\}$,
$p^+_m= \sup_{x\in \Omega}p_m(x),$ and $p^+_M= \sup_{x\in \Omega}p_M(x).$
The anisotropic variable exponent Sobolev space $W^{1,\overrightarrow{p}(x)}(\Omega)$ consists of functions $u\in L^{p_M(x)}(\Omega)$ such that $\partial_{x_i}u\in L^{p_i(x)}(\Omega)$ for all $i\in\{1,\dots,N\}$. This space can also be defined as
\begin{equation*}
	\begin{aligned}
W^{1,\overrightarrow{p}(x)}(\Omega)=\{u\in L^{1}_{loc}(\Omega) : u \in L^{p_i(x)}(\Omega) \text{ and }\partial_{x_i}u\in L^{p_i(x)}(\Omega)~ \text{for all} ~1\leq i\leq N\},
\end{aligned}
\end{equation*}
and it is equipped with the norm
$$
\|u\|_{W^{1,\overrightarrow{p}(x)}(\Omega)}:=\|u\|_{L^{p_M(x)}(\Omega)}+\sum_{i=1}^N\left\|\partial_{x_i}u \right\|_{L^{p_i(x)}(\Omega)} \text{ for all } u \in W^{1,\overrightarrow{p}(x)}(\Omega).
$$
The space $\big(W^{1,\overrightarrow{p}(x)}(\Omega),\|\cdot\|_{W^{1,\overrightarrow{p}(x)}(\Omega)}\big)$ forms a reflexive Banach space, as proven by Fan \cite[Theorems 2.1 and 2.2]{fananiso}.

The anisotropic variable exponent Sobolev space with zero boundary values $W^{1,\overrightarrow{p}(x)}_0(\Omega)$ is defined as the closure of $C_{0}^{\infty}(\Omega)$, with the norm
$
\|u\|_{\overrightarrow{p}(x)}
:=\sum_{i=1}^N\left\|\partial_{x_i}u \right\|_{L^{p_i(x)}(\Omega)}.
$
Furthermore, the space $W^{1,\overrightarrow{p}(x)}_0(\Omega)$ allows for the appropriate treatment of the existence of weak solutions for problem \eqref{e1.1} and can be considered a natural generalization of the variable exponent Sobolev space $W^{1, p(x)}_0(\Omega)$. On the other hand, the space $W^{1,\overrightarrow{p}(x)}_0(\Omega)$ can also be regarded as a natural generalization of the classical anisotropic Sobolev space $W^{1,\overrightarrow{p}}_0(\Omega)$, where $\overrightarrow{p}$ is the constant vector $(p_1,\dots,p_N)$.\par
In the sequel, we shall present a revised version of the critical Sobolev embedding theorem tailored to anisotropic variable exponent Sobolev spaces.
\begin{theorem} \label{theo23}
	Let $p_i \in C_{+}(\overline{\Omega})$ for all $i\in \{1,\dots,N\}$, with $p_{m}\in C_{+}^{\log}(\overline{\Omega})$ such that $p_m^{+}\leq N$. Suppose that $h\in C(\overline{\Omega})$ satisfies the condition
	$ 1\leq h(x) \leq p_m^{\ast}(x) \text{ for all } x\in \overline{\Omega}.$	
	Then, there exists a continuous embedding
	$W^{1,\overrightarrow{p}(x)}(\Omega)\hookrightarrow L^{h(x)}(\Omega).$
	If in addition, we assume that $1\leq h(x)< p_m^{\ast}(x)$ for all $x\in \overline{\Omega}$, then this embedding is also compact.
\end{theorem}
\begin{proof}
	Let $u\in W_{}^{1,\overrightarrow{p}(x)}(\Omega)$. According to Proposition \ref{prop21}, we can conclude that $u\in W^{1,p_m(x)}(\Omega)$. Since $h(x)\leq p_m^{\ast}(x)$ for all $x\in \overline{\Omega}$, Proposition \ref{prop24} guarantees the existence of a positive constant $c>0$ such that
	\begin{equation}\label{eq2.1}
		\|u\|_{L^{h(x)}(\Omega)} \leq  c\Big(
		\|u\|_{L^{p_m(x)}(\Omega)}  + \sum_{i=1}^{N}\|\partial_{x_i}u\|_{L^{p_m(x)}(\Omega)} 	\Big).
	\end{equation}
	Since $p_m \leq p_i \leq p_M$ holds for all $i\in \{1,\dots,N\}$, we can again apply Proposition \ref{prop21} to find positive constants $c_i$ such that
	\begin{equation}\label{eq2.2}
		\|u\|_{L^{p_m(x)}(\Omega)} \leq 	c_0\|u\|_{L^{p_M(x)}(\Omega)}~~~~ \text{ and }~~~~
		\|\partial_{x_i}u\|_{L^{p_m(x)}(\Omega)} \leq
		c_i\|\partial_{x_i}u\|_{L^{p_i(x)}(\Omega)},
	\end{equation}
	for all  $i \in \{1,\dots,N\}$.  By combining \eqref{eq2.1} with \eqref{eq2.2}, we obtain
	$\|u\|_{L^{h(x)}(\Omega)} \leq c \|u\|_{W^{1, \overrightarrow{p}(x)}(\Omega)}.$
	Since Proposition \ref{prop24} establishes that the embedding
	$$W^{1, p_m(x)}(\Omega)\hookrightarrow L^{h(x)}(\Omega),$$
	 is compact if $1\leq h(x)< p_m^{\ast}(x)$ for all $x\in \Omega$, we can conclude that the embedding $W^{1, \overrightarrow{p}(x)}(\Omega)\hookrightarrow L^{h(x)}(\Omega)$ is both continuous and compact, thus completing the proof of Theorem \ref{theo23}. 
\end{proof}
\begin{remark}\label{remark 2}
	$1)$ The conclusions of Theorem \ref{theo23} remain valid in a more general context. Specifically, one can extend theorem's applicability by replacing the "critical exponent" $p_{m}^{\ast}(x)$ with the function $\widehat{p}_{m,M}^{\ast }(x)=\{p_{m}^{\ast}(x),p_{M}(x)\}$.\par
	2) It is worth noting that when $p_1(x)=p_2(x)=\dots= p_N(x)=p(x)$ and $p(x)<N$ holds for all $x\in \overline{\Omega}$, the "critical exponent" $p_m^{\ast}(x)$ in Theorem \ref{theo23} coincides with the "critical exponent" $p^{\ast}(x)$ for $W^{1,p(x)}(\Omega)$, as can be seen in Proposition \ref{prop24}.\par
	3) Ji \cite{ji} conducted a study of anisotropic equations in the subcritical case, using the "critical exponent" $\Big(\textbf{P}_{-}^{-} \Big)^{\ast}$, where $\textbf{P}_{-}^{-}=\inf\{p_1^{-},p_2^{-},\dots,p_N^{-}\}$. Our Theorem \ref{theo23} improves upon the results obtained in Ji \cite{ji} by replacing the critical exponent $\Big(\textbf{P}_{-}^{-} \Big)^{\ast}$ with $p_m^{\ast}(x)$.
\end{remark}
\begin{definition}
	Let $E:X\to \mathbb{R}$ be a $C^1$ function defined on a real Banach space $X$. A sequence $\{u_n\}$ is termed a Palais-Smale sequence (abbreviated as (PS)-sequence) on $X$ if it satisfies the following conditions:
	
	1. $E(u_n)$ is bounded.\par
	2. $E'(u_n)\to 0$ in the dual space $X'$.
	
	If, in addition to the above conditions, $E(u_n)$ converges to a finite value $c\in \mathbb{R}$ as $n$ tends to infinity, then the (PS)-sequence is referred to as a $(PS)_c$-sequence.
	
	Furthermore, if every $(PS)_c$-sequence for the function $E$ possesses a subsequence that converges strongly in $X$, then we say that $E$ satisfies the Palais-Smale condition at level $c$ (or $E$ is $(PS)_c$, for brevity).
\end{definition}\par 
 We shall conclude this section by presenting two classical theorems: the Mountain Pass Theorem and its Rabinowitz $\mathbb{Z}_2$–symmetric version. These theorems will play a crucial role in proving our main results in Section \ref{sec 4}. The theorems are summarized below.
\begin{theorem}[Ambrosetti and Rabinowitz  {\cite[Theorem 2.1]{AmbrosettiRabinowitz}}] \label{PMT}
	Consider a real infinite-dimensional Banach space $X$ and let $E: X \to \mathbb{R}$ be a $C^{1}$ function satisfying the $(PS)_c$ condition with $E(0_X)=0$. Assume the following conditions :
	\begin{itemize}
		\item[($\mathcal{I}_1$)] There exist positive constants $\mathcal{R}$ and $\rho$ such that $E(u)\geq\mathcal{R}$ for all $u\in X$ with $\|u\|_{X}=\rho$.
		\item[($\mathcal{I}_2$)] There exists an element $\tilde{u}\in X$ such that  $\left\| \tilde{u}\right\|_{X} >\rho$ and $E(\tilde{u})<0$.
	\end{itemize}
	Then. $E$ has a critical value $c\geq \mathcal{R}$, which can be characterized as
	\[c := \inf_{\phi \in \Gamma} \max_{ t \in \left[ 0,1\right]  }E(\phi(t)),\]
	where
	$$\Gamma= \big\{ \phi \in C\big( \left[ 0,1\right],X\big) : \phi(0)=0_X, E(\phi (1))<0 \big\} . $$
\end{theorem}\par 
This theorem provides conditions under which a function $E$ has a critical value $c$, and it characterizes $c$ as the infimum of a certain set of functions $\phi$ in $X$.
\begin{theorem}[Rabinowitz  {\cite[Theorem 9.12]{Rabinowitz}}] \label{SPMT}
	Let $X$ be a real infinite-dimensional Banach space and let $E: X \to \mathbb{R}$ be even and
	of class $C^{1}$, satisfying $(PS)_c$ and  $E(0_X)=0$. Suppose that assumption ($\mathcal{I}_1$) holds in addition to  condition
	($\mathcal{I}_2 '$):
	For all finite-dimensional subspaces $X_1 \subset X$, the set
	$\mathcal{S}_1:= \{ u \in X_1: E(u) \geq 0\}$ is bounded in $X$.
	Then  $E$ has has an unbounded sequence of critical values.
\end{theorem}\par 
This theorem provides conditions under which a function $E$ has an unbounded sequence of critical values, where critical values are defined with respect to the $(PS)_c$ condition.\\

\textbf{\textbf{Notations.}} In our discussions, we shall use the following notations: Strong (resp. weak, weak-*) convergence is denoted by $\rightarrow$ (resp., $\rightharpoonup$, $\overset{\ast}{\rightharpoonup}$),
constants: $C_i$, $C_{i}'$, $c_j$, and $c_{i}"$ represent positive constants, which may vary from one line of the text to another and can be determined under specific conditions. $X^{\ast}$ denotes the dual space of $X$, $\delta_{x_j}$ represents the Dirac mass at the point $x_j$. For any $\rho>0$ and $x\in \Omega$, $B(x,\rho)$ denotes the open ball of radius $\rho$ centered at $x$, the characteristic function of a set $B$ is denoted by $\chi_B$.

\section{An extension of the Lions concentration-compactness principle} \label{sec 3}
In this section, we shall establish the extension of the  concentration-compactness principle to the anisotropic variable exponent Sobolev spaces, which is one of the main results in this paper.

In what follows, we shall
denote by $\mathcal{M}(\overline{\Omega})$ the class of nonnegative Borel measures of finite total mass, and a sequence $\mu_{n} \overset{\ast}{\rightharpoonup} \mu $ in $\mathcal{M}(\overline{\Omega})$ if and only if $\displaystyle \int_{\Omega} \phi(x) d\mu_{n} \to \int_{\Omega} \phi(x) d\mu_{}$ for every test function $\phi \in C^{\infty}(\Omega)\cap C(\overline{\Omega})$.
Note that by Theorem \ref{theo23}, we have
\begin{equation}\label{24}
	S_{h}  :=\inf_{u\in W_0^{1,\overrightarrow{p}(x)}(\Omega)\backslash \{0\}}
	\frac{ \|u\|_{\overrightarrow{p}(x)}}{\| u \|_{L^{h(x)}(\Omega)}}>0.
\end{equation}
We now state the main result of this section, that is, a concentration-compactness principle for the anisotropic variable exponent Sobolev spaces.

\begin{theorem}\label{ccp}
	Consider continuous functions $p_i$ and $h$ on $\Omega$ such that $1 < \inf_{x\in\Omega}p_i(x) \leq \sup_{x\in\Omega}p_i(x) < N$ for all $i \in \{1,2,\dots,N\}$ and $1 \leq h(x) \leq p_{m}^{\ast}(x)$ in $\Omega$, where $p_{m}\in C_{+}^{\log}(\overline{\Omega})$.  Let $\{u_n\}_{n\in\mathbb{N}}$ be a weakly convergent sequence in
	$W_0^{1,\overrightarrow{p}(x)}(\Omega)$ with weak limit $u$, and such that
	$| \partial_{x_i} u_n|^{p_i(x)}\overset{\ast}{\rightharpoonup}\mu_i$
	in $\mathcal{M}(\overline{\Omega})$
	and
	$|u_n|^{h(x)}\overset{\ast}{\rightharpoonup}\nu$ in $\mathcal{M}(\overline{\Omega})$.
	Also, suppose that the set $\mathcal{A} = \{x\in \Omega\colon h(x)=p_{m}^{\ast}(x)\}$ is nonempty. Then  there exist  $\{x_j\}_{j\in J}\subset \mathcal{A}$ of distinct points and   $\{\mu_j\}_{j\in J}$, $\{\nu_j\}_{j\in J}\subset (0,\infty)$, where $J$ is countable index set, such that
	\begin{gather}
		\nu=|u|^{h(x)} + \sum_{j\in J}\nu_j\delta_{x_j},\    \label{21}\\
		\mu \geq \sum_{i=1}^{N}|\partial_{x_i} u|^{p_i(x)} + \sum_{j\in J} \mu_j \delta_{x_j}, \        \label{22}\\
		\displaystyle
		N^{1-p_M^{+}} S_h	\nu_j^{\frac{1}{p_m^{\ast}(x_j)}} \leq
		\max \Big\{ \Big(  \mu_j   \Big) ^{1/p_{M}^{+}}
		, \Big( \mu_j  \Big) ^{1/p_{m}^{-}}
		\Big\}.\   
		\hbox{ for all } 
		 j\in J.    \label{23}
	\end{gather}
	where $\delta_{x_j}$ is the Dirac mass at $x_j$ and $ \displaystyle\mu= \sum_{i=1}^{n}\mu_i$.
\end{theorem}\par 
Before we give the proof of Theorem  \ref{ccp}, we recall some auxiliary results obtained by Bonder and Silva \cite{Bonder}.
\begin{lemma}[Bonder and Silva {\cite[Lemma 3.1]{Bonder}}]\label{conv.debil}
	Let $\nu , \{\nu_n\}_{n\in\mathbb{N}} \in \mathcal{M}(\overline{\Omega})$ be such that $\nu_n\overset{\ast}{\rightharpoonup}\nu$ in $\mathcal{M}(\overline{\Omega})$. Then  for all  $q\in C_{+}(\Omega)$, we have
	$
	\|\psi\|_{L^{q(x)}_{\nu_n}(\Omega)} \to
	\|\psi\|_{L^{q(x)}_{\nu}(\Omega)} \   \text{as } n\to \infty, \text{ for all } \psi\in C^{\infty}(\overline{\Omega}).
	$
\end{lemma}
\begin{lemma}[Bonder and Silva {\cite[Lemma 3.2]{Bonder}}]\label{Lema 1ccp}
	Let $\mu,\nu \in \mathcal{M}(\overline{\Omega})$ be such that there is some positive constant $c,$ such that
	$$
	\|\psi\|_{L_\nu^{h(x)}(\Omega)}\leq c\|\psi\|_{L_\mu^{p(x)}(\Omega)},$$
	 for some  $p,h \in C_{+}(\overline{\Omega}) \text{ satisfying } \inf_{x\in \Omega}(h(x)-p(x))>0.
	$
	Then  there is an at most countable set $\{x_j\}_{j\in J}$ of distinct points in $\overline{\Omega}$ and
	$\{\nu_j\}_{j\in J}\subset (0,\infty)$, such that
	$
	\nu=	\displaystyle \sum\limits_{j\in J} \nu_j\delta_{x_j}.
	$
\end{lemma} \par 
The following lemma is an extension of the Brezis–Lieb Lemma to Lebesgue spaces with variable exponents.
\begin{lemma}[Bonder and Silva {\cite[Lemma 2.1]{Fu}}]\label{lema 2ccp}
	Consider a bounded sequence $\{u_n\}_{n\in \mathbb{N}}$ in $L^{h(x)}(\Omega)$ and let $u(x)$ be such that $u_n(x)$ converges to $u(x)$ in $L^{h(x)}(\Omega)$ for a.e. $x\in\Omega$. Then  the following holds:
	\[
	\lim_{n\to\infty}\left(\int_\Omega|u_n|^{h(x)}dx - \int_\Omega|u-u_n|^{h(x)}dx\right) = \int_\Omega|u|^{h(x)}dx.
	\]
\end{lemma}
\begin{proof}[Proof of Theorem \ref{ccp}]
	We start by establishing relation \eqref{21}. To this end, we put  $v_n=u_n-u$.  Then we have up to a subsequence, that
	\begin{equation}\label{convn}
		\begin{gathered}
			v_n \rightharpoonup 0 \text{ in } W_{0}^{1,\overrightarrow{p}(x)}(\Omega), \text{ and }
			v_n(x) ~~\to ~~ 0,  \text{ for a.e. }  x\in \Omega.
		\end{gathered}
	\end{equation}
	So, by the Brezis-Lieb Lemma \ref{lema 2ccp},  we can see that
	$
	\lim_{n\to\infty}\int_\Omega\Big||u_n|^{h(x)}
	-|v_n|^{h(x)} -|u|^{h(x)}\Big|dx=0.
	$
	Thus, from the last equality and  relation \eqref{convn}, we can deduce that
	$
	\lim_{n\to\infty}\Big(\int_\Omega \psi|u_n|^{h(x)}
	-\int_\Omega\psi|v_n|^{h(x)} dx\Big)
	=\int_\Omega\psi|u|^{h(x)} dx, \   \text{ for all } \psi \in C(\overline{\Omega}),
	$
	that is,
	\begin{equation} \label{L336}
		|v_n|^{h(x)} \overset{\ast}{\rightharpoonup}  \nu-|u|^{h(x)} = \tilde{\nu} \text{ in }  \mathcal{M}(\overline{\Omega}).
	\end{equation}		
	Obviously,
	$\Big\{\displaystyle\sum_{i=1}^{n}|\partial_{x_i}u_n|^{p_i(x)}\Big\}	$ is bounded in $L^1(\Omega)$.  So up to a subsequence, we have  as $ n\to +\infty$
	\begin{equation} \label{L337}
		\sum_{i=1}^{n}|\partial_{x_i}u_n|^{p_i(x)} \overset{\ast}{\rightharpoonup} \tilde{\mu} \text{ in } \mathcal{M}(\overline{\Omega}) \text{ for some } \tilde{\mu}=	\sum_{i=1}^{n}\tilde{\mu}_i\in \mathcal{M}(\overline{\Omega}).
	\end{equation}
	Clearly, $\psi v\in  W_{0}^{1,\overrightarrow{p}(x)}(\Omega)$ for all   $\psi \in C^{\infty}(\overline{\Omega})$ and  $v\in  W_{0}^{1,\overrightarrow{p}(x)}(\Omega)$. Then  by applying relation \eqref{24}, for all  $\psi \in C^{\infty}(\overline{\Omega})$, we obtain
	\begin{align*}
		\displaystyle S_h\| \psi v_n\|_{L^{h(x)}(\Omega)} &\leq \sum_{i=1}^{N}\| \partial_{x_i}
		(\psi v_n)\|_{L^{p_i(x)}(\Omega)}\\
		&\leq \sum_{i=1}^{N}\big(\| \psi\partial_{x_i}
		v_n\|_{L^{p_i(x)}(\Omega)} + \|v_n \partial_{x_i}
		\psi\|_{L^{p_i(x)}(\Omega)}\big)\\
		&\leq \sum_{i=1}^{N}\big(\| \psi\partial_{x_i}
		v_n\|_{L^{p_i(x)}(\Omega)} + \|\partial_{x_i}
		\psi\|_{L^{\infty}(\Omega)}\|v_n \|_{L^{p_i(x)}(\Omega)}\big) .
	\end{align*}
	Since  $v_n \rightharpoonup  0$ in $W_{0}^{1,\overrightarrow{p}(x)}(\Omega)$ (according to relation \eqref{convn}), we can infer that $v_n \to 0$ in $L^{p_i(x)}(\Omega),$ for all  $i\in\{1,2,\dots,N\},$ in view of Theorem \ref{theo23}.
	By Lemma \ref{conv.debil}, we get
	$$
	S_h\| \psi \|_{L_{\tilde{\nu}}^{h(x)}(\overline{\Omega})}
	\leq
	\sum_{i=1}^{N}\| \psi\|_{L_{\tilde{\mu}_i}^{p_i(x)}(\overline{\Omega})} \leq N \max_{1\leq i \leq N}\| \psi\|_{L_{\tilde{\mu}_i}^{p_i(x)}(\overline{\Omega})},
	$$	
	for all  $\psi \in C^{\infty}(\overline{\Omega})$. Hence by applying Lemma \ref{Lema 1ccp}, we obtain
	$\displaystyle \tilde{\nu}= \sum_{j\in I}\nu_j\delta_{x_j},$
	for some at most countable set $J$, a family of $\{x_j\}_{j\in J}\subset\Omega$ and a family of nonnegative numbers$\{\nu_j\}_{j\in J}$. That is, we have thus obtained relation \eqref{21}.\par
	Let us now prove that the points $x_j$ actually belong to the critical set $\mathcal{A}$. Assume to the contrary, that there exist some  $x_j \in$ in $ \overline{\Omega} \backslash \mathcal{A}$. Let $\rho$ be a positive number such that  $\overline{B_{}(x_j,2\rho)}\subset \mathbb{R}^N\backslash \mathcal{A}$, noting the closedness of $\mathcal{A}$. We put  $\Omega_{B_j}=B_{}(x_j,\delta)\cap \overline{\Omega}$, and get  $\overline{\Omega_{B_j}} \subset \overline{\Omega} \backslash \mathcal{A}$ and hence, $h(x)< p_m^{\ast}(x)-\varepsilon $, for some $\varepsilon>0$ in $\overline{\Omega_{B_j}}$. Since  $1<h(x)< p_m^{\ast}(x)-\varepsilon $ for each $ x\in \overline{\Omega_{B_j}}\cap \overline{\Omega}$, we can get $\bar{h}\in C_{+}(\overline{\Omega})$ such that $\bar{h}|_{\overline{\Omega_{B_j}}}=h$ and $\bar{h}(x)<p_m^{\ast}(x)-\varepsilon $ for all  $x\in \overline{\Omega}$. Therefore by Theorem \ref{theo23}, we find $ u_n \to u$ in $L^{\bar{h}(x)}(\Omega)$. Equivalently, $\displaystyle \int_{\Omega_{B_j}}|u_n-u|^{\bar{h}(x)}dx \to 0$.
	Thus, by applying the Brezis-Lieb Lemma \ref{lema 2ccp}, we find
	$ \displaystyle \int_{\Omega_{B_j}}|u_n|^{h(x)}dx \to \int_{\Omega_{B_j}}|u|^{h(x)}dx.$	
	Hence, from this and the fact that
	$ \displaystyle \nu(\Omega_{B_j}) \leq \liminf_{n\to \infty}\int_{\Omega_{B_j}}|u_n|^{h(x)}dx$ (see Fonseca and Leoni \cite[Proposition 1.203]{Fonseca}),  we get
	$\displaystyle\nu(\Omega_{B_j}) \leq \int_{\Omega_{B_j}}|u|^{h(x)}dx.$	
	It follows from relation \eqref{21} that
	$\nu(\Omega_{B_j}) \geq \int_{\Omega_{B_j}}|u|^{h(x)}dx + \nu_j > \int_{\Omega_{B_j}}|u|^{h(x)}dx,$
	which is a contradiction, hence  $\{x_j\}_{j\in J}\subset \mathcal{A}$.
	
	Next, to obtain relation \eqref{22}, consider $\psi \in C_{c}^{\infty}(\mathbb{R}^N)$   such that $\chi_{B(0, \frac{1}{2})} \leq \psi \leq \chi_{B(0,1)}$. Suppose that $J\neq \emptyset$ and fix $j\in J$. Consider $\varepsilon_i>0$ to be arbitrary for all $i\in \{1,2,\dots,N\}$.  Set
	$\displaystyle\psi_{j}(x):=\psi\big(\frac{x_1 -x_{j,1}}{\varepsilon_{1}},\dots,\frac{x_N -x_{j,N}}{\varepsilon_{N}}\big)$.	
	By again using relation \eqref{24} to $\psi_{j} u_n$, we have	
	\begin{align*}
		\displaystyle S_h\| \psi_{j} u_n\|_{L^{h(x)}(\Omega)} &\leq \sum_{i=1}^{N}\big(\| \psi_{j}\partial_{x_i}
		u_n\|_{L^{p_i(x)}(\Omega)} + \|u_n \partial_{x_i}
		\psi_{j}\|_{L^{p_i(x)}(\Omega)}\big)\\
		&\leq   \sum_{i=1}^{N}\bigg(\| \psi_{j}\partial_{x_i}
		u_n\|_{L^{p_i(x)}(\Omega)} + \|u \partial_{x_i}
		\psi_{j}\|_{L^{p_i(x)}(\Omega)}\\
	&\qquad	+	\|\partial_{x_i}
		\psi_{j}\|_{L^{\infty}(\Omega)}
		\|u_n -u\|_{L^{p_i(x)}(\Omega)}
		\bigg).
	\end{align*}	
	Then  by using \eqref{convn} and Theorem \ref{theo23}  and  taking $n\to \infty$ in the last estimate, we obtain	
	\begin{align} \label{Li37}
		\displaystyle S_h\| \psi_{j} \|_{L_{\nu}^{h(x)}(\Omega)} &\leq
		\sum_{i=1}^{N}\| \psi_{j}
		\|_{L_{\mu_i}^{p_i(x)}(\Omega)}  +  \sum_{i=1}^{N} \|u \partial_{x_i}
		\psi_{j}\|_{L^{p_i(x)}(\Omega)}.
	\end{align}
	On the one hand, by using relations \eqref{L24}, \eqref{L25} and Remark \ref{remark 1}, we have
	\begin{align*}
		&\displaystyle \| \psi_{j} \|_{L_{\nu}^{h(x)}(\Omega)} 
		\geq
		\min\Bigg\{  \Big(\int_{B(x_j,\displaystyle\max_{1\leq i \leq N}\varepsilon_i)}|\psi_{j}|^{h(x)}d\nu \Big)^{\frac{1}{h_{j,\varepsilon}^+}},
		\Big(\int_{B(x_j,\displaystyle\max_{1\leq i \leq N}\varepsilon_i)}|\psi_{j}|^{h(x)}d\nu \Big)^{\frac{1}{h_{j,\varepsilon}^-}}
		\Bigg\},
	\end{align*}
	where
	\begin{gather*}
		h^+_{j,\varepsilon}:=\sup_{B\big(x_j,\displaystyle\max_{1\leq i \leq N}\varepsilon_i\big)}h(x) \   \text{ and }\
		h^-_{j,\varepsilon}:=\inf_{B\big(x_j,\displaystyle\max_{1\leq i \leq N}\varepsilon_i\big)}h(x).
	\end{gather*}
	Since
	$$
	\displaystyle \int_{B\big(x_j,\displaystyle\max_{1\leq i \leq N}\varepsilon_i\big)}|\psi_{j}|^{h(x)}d\nu \geq \int_{B\Big(x_j,\frac{\displaystyle\max_{1\leq i \leq N}\varepsilon_i}{2}\Big)}|\psi_{j}|^{h(x)}d\nu\geq \nu(\{x_j\})=\nu_j,
	$$
	it follows that
	$
	\| \psi_{j} \|_{L_{\nu}^{h(x)}(\Omega)} \geq
	\min\big\{  \nu_j^{\frac{1}{h_{j,\varepsilon}^+}},
	\nu_j^{\frac{1}{h_{j,\varepsilon}^-}}
	\big\}.
	$
	When letting  $\displaystyle\max_{1\leq i \leq N}\varepsilon_i \to 0^+$ in the above inequality and using  the fact $x_j \in \mathcal{A}$ and the continuity of $h$, we get
	\begin{align}\label{L3.8}
		\displaystyle \| \psi_{j} \|_{L_{\nu}^{h(x)}(\Omega)}   &\geq  \nu_j^{\frac{1}{p_m^*(x_j)}}.
	\end{align}
	On the other hand,  by relations \eqref{L24}, \eqref{L25} and using Jensen's inequality on the convex function $q: \mathbb{R}^{+}\to  \mathbb{R}^{+}, q(t)=  t^{\bar{p_{m,M}}}$, $\tilde{p_{m,M}}> 1$, we see that
	\begin{equation}\label{Jensen}
			\begin{aligned}
		\displaystyle
		\frac{\|\psi_{j}\|_{\overrightarrow{p}(x)}^{\bar{p_{m,M}}}}{N^{\bar{p_{m,M}}-1}}  &=
		N\Bigg(\displaystyle\frac{\sum_{i=1}^{N}\| \psi_{j} \|_{L_{\mu_i}^{p_i(x)}(\Omega)}}{N}\Bigg)^{\bar{p_{m,M}}}
		\leq	\sum_{i=1}^{N}\| \psi_{j}
		\|_{L_{\mu_i}^{p_i(x)}(\Omega)}^{p_M^{+}}\\
		&\leq	\sum_{i=1}^{N}\| \psi_{j}
		\|_{L_{\mu_i}^{p_i(x)}(\Omega)}^{p_i^{+}}
		\leq 	\sum_{i=1}^{N} \int_{\Omega} |\psi_{j}|^{p_i(x)}d\mu_i,
		\end{aligned}
	\end{equation}
	where
	$ \bar{p_{m,M}}= p_{M}^{+}$ if $   \|\psi_{j}\|_{\overrightarrow{p}(x)}<1$
	and
	$\bar{p_{m,M}}= p_{m}^{-}$ if $\|\psi_{j}\|_{\overrightarrow{p}(x)}\geq 1$.
	Since  $ \psi_{j}^{p_i(x)} \leq \psi_{j}$, we have
	$$
	\displaystyle   \lim_{\displaystyle \max_{1\leq i \leq N}\varepsilon_i\to 0}
	\sum_{i=1}^{N} \int_{\Omega} |\psi_{j}|^{p_i(x)}d\mu_i
	\leq   \lim_{\displaystyle \max_{1\leq i \leq N}\varepsilon_i\to 0}\int_{\displaystyle B(x_j,\displaystyle\max_{1\leq i \leq N}\varepsilon_i)} \psi_{\varepsilon,j}d\mu,$$
	where
	$
	\mu =	\sum_{i=1}^{N}\mu_i
	$
	$ 	
	\leq  \lim_{\displaystyle \max_{1\leq i \leq N}\varepsilon_i\to 0} \mu (\displaystyle B(x_j,\displaystyle\max_{1\leq i \leq N}\varepsilon_i))
	= \mu_j.
	$
	From this and  relation \eqref{Jensen}, we get
	\begin{equation}\label{L3.9}
		\sum_{i=1}^{N}\| \psi_{j}
		\|_{L_{\mu_i}^{p_i(x)}(\Omega)} \leq  \max \Big\{ N^{p_M^{+}-1}\Big(  \mu_j   \Big) ^{1/p_M^{+}}
		, N^{p_m^{-}-1}\Big( \mu_j  \Big) ^{1/p_m^{-}}
		\Big\}.
	\end{equation}
	Next, we shall prove that $\displaystyle\sum_{i=1}^{N} \|u \partial_{x_i}
	\psi_{j}\|_{L^{p_i(x)}(\Omega)} \to 0,$ as $\displaystyle \max_{ \underset{1\leq i \leq N}{}}\varepsilon_i \to 0^{+}$. Indeed, applying the H\"older inequality, we obtain
		\begin{equation*}
		\begin{aligned}
			\displaystyle
			\int_{\Omega} |u\partial_{x_i}
			\psi_{j}|^{p_i(x)} \, dx  &=\int_{\Omega} |u|^{p_i(x)}|\partial_{x_i}
			\psi_{j}|^{p_i(x)} \, dx
			\leq	
		\| |u|^{p_i(x)}\|_{L^{\frac{N}{N-p_i (x)}}(B(x_j,\varepsilon_i))}
		\||\partial_{x_i}
		\psi_{j}|^{p_i(x)} \|_{L^{\frac{N}{p_i(x)}}
			(B(x_j,\varepsilon_i))}.
		\end{aligned}
	\end{equation*}
	Moreover, by  relation \eqref{L225} and using  $$\partial_{x_i}
	\psi_{j}=\frac{1}{\varepsilon_i}\partial_{x_i}\psi\left(\frac{x_1-x_{j,1}}{\varepsilon_1},\dots,\frac{x_N-x_{j,N}}{\varepsilon_N}\right),$$ we obtain the following
		\begin{equation*}
		\begin{aligned}
			\displaystyle \big\||\partial_{x_i}
			\psi_{j}|^{p_i(x)} \big\|_{L^{\frac{N}{p_i(x)}}
				(B(x_j,\varepsilon_i))}^{p_i^{-}}
			 &\leq 1 +\int_{B_{}(x_j,\varepsilon_i)} |\partial_{x_i}
			 \psi_{j}|^N\\
			  &\leq 1+ \int_{B_{}(x_j,\varepsilon_i)} |\partial_{x_i}
			 \psi_{j}|^N \frac{1}{\varepsilon_i^N}\, dx
			\\
			&\leq	1 +\int_{B(0,1)} |\partial_{x_i}
	\psi(y)|^N\, dy,
		\end{aligned}
	\end{equation*}
	so we can conclude from the last two estimates that
	$$
	\int_{\Omega} |u\partial_{x_i}
	\psi_{j}|^{p_i(x)} \, dx \to 0 ~~~~\text{ as }   \displaystyle \max_{1\leq i \leq N}\varepsilon_i \to 0^{+}\to 0,
	$$
	hence, from relation \eqref{L26}, we get
	\begin{equation} \label{L3.10}
		\displaystyle\sum_{i=1}^{N}\|u \partial_{x_i}
		\psi_{j}\|_{L^{p_i(x)}(\Omega)} \to 0 \
		\text{ as }   \displaystyle \max_{ 1\leq i \leq N}\varepsilon_i \to 0^{+}.
	\end{equation}
	Letting $\displaystyle \max_{ 1\leq i \leq N}\varepsilon_i \to 0^{+}$ in \eqref{Li37} and taking into consideration \eqref{L3.8},\eqref{L3.9} and \eqref{L3.10}, we get
	$$
	N^{1-p_M^{+}} S_h	\nu_j^{\frac{1}{p_m^*(x_j)}}
	\leq
	\max \Big\{ \Big(  \mu_j   \Big) ^{1/p_{M}^{+}}, \Big( \mu_j  \Big) ^{1/p_{m}^{-}}
	\Big\}.
	$$
	This shows that $\{x_j\}_{j\in J}$ are all atomic points of $\mu$.
	Finally, to obtain relation \eqref{22}, we notice that for each $\psi\in  C(\Omega)$ with $\psi \geq 0$, the functional
	$ u \mapsto \displaystyle\int_{\Omega}\psi(x)  \bigg(\sum_{i=1}^{N}|\partial_{x_i}u|^{p_i(x)}\bigg)dx
	$
	is convex and differentiable on $W_0^{1,\overrightarrow{p}(x)}(\Omega)$. Therefore it is weakly lower semicontinuous and we obtain\\
		\begin{equation*}
		\begin{aligned}
			\int_{\Omega}\psi(x)  \bigg(\sum_{i=1}^{N}|\partial_{x_i}u|^{p_i(x)}\bigg)dx
			&\leq \liminf_{n\to \infty} \displaystyle\int_{\Omega}\psi(x)  \bigg(\sum_{i=1}^{N}|\partial_{x_i}u_n|^{p_i(x)}\bigg)dx\\
			&=\int_{\overline{\Omega}} \psi d \mu .
		\end{aligned}
	\end{equation*}
	Hence
	$ \mu \geq \displaystyle \sum_{i=1}^{N}|\partial_{x_i}u|^{p_i(x)}.$
	Extracting $\mu$ to its atoms, we get relation \eqref{22} and the proof of Theorem \ref{ccp} is complete. 
\end{proof}
\section{ A class of nonlinear anisotropic elliptic equations with critical growth} \label{sec 4}
In this section, we shall establish the existence and multiplicity of nontrivial solutions for problem \eqref{e1.1}. Throughout this paper, we assume that $f$ satisfies the following conditions:	
\begin{itemize}
	\item[$(\textbf{\textit{f}}_1)$] There exist a positive function $\ell\in C(\overline{\Omega})$ and a positive constant $C_{f}$  such that
	$| f(x,\xi)|\leq C_{f}(1+| \xi |^{\ell(x)-1}), \  \text{for all } (x,\xi)\in\Omega\times\mathbb{R},$
	where $p^{+}_{M}<\ell(x)<h(x)\leq p_m^*(x)$ for all $x\in\overline{\Omega}$.
	\item[$(\textbf{\textit{f}}_2)$] There exist  $R>0$ and $\theta_{\lambda}\geq r^{+}$ (resp. $\theta_{\lambda}\leq r^{-}$ ) if $\lambda\geq 0$ (resp. $\lambda< 0$ ), such that for all   $\xi$ with $| \xi|\geq R$ and $ x\in\Omega$, we have
	$0<\theta_{\lambda}f(x,\xi)\leq \xi f(x,\xi) .$
	\item[($\textbf{\textit{f}}_3$)]$f(x,\xi)=\circ(| \xi |^{p^{+}_{M}})$ as $\xi\to 0$ and uniformly for all $x\in\Omega$.
	\item[$(\textbf{\textit{f}}_4)$] $f$ is odd in $\xi$, that is, $f(x,-\xi)=-f(x,\xi)$, for all $(x,\xi)\in\Omega\times\mathbb{R}$.
	\item[$(\textbf{\textit{H}})$]  $p^{+}_{M}< r(x)<h(x)\leq p_m^*(x)$ and $\mathcal{A} = \{x\in\Omega \colon h(x)=p_m^*(x)\}\neq\emptyset$.
\end{itemize}
\par
Throughout this article, for simplicity, we denote the anisotropic variable exponent space $W_0^{1,\overrightarrow{p}(x)}(\Omega) $ by $X$.
\begin{definition}
	We say that $u\in X$ is a weak solution of problem \eqref{e1.1} if
	\begin{equation}
		 		\int_{\Omega}\sum_{i=1}^Na_i(x,\partial_{x_i}u)\partial_{x_i}v\, \mathrm{d}x + \lambda \int_{\Omega} | u|^{r(x)-2} uv\, \mathrm{d}x
			=  \int_{\Omega} | u|^{h(x)-2} uv\, \mathrm{d}x 
			\	+\beta \int_{\Omega} f(x,u)v\, \mathrm{d}x,
			\end{equation}
	for all
	$
	v \in X.
	$
\end{definition}
The energy functional associated with problem \eqref{e1.1} is defined by $E_{\lambda,\beta} : X\to \mathbb{R}$, where
\begin{equation} \label{e2.1}	
		E_{\lambda,\beta}(u)=\int_{\Omega}\sum_{i=1}^N
		A_i(x,\partial_{x_i}u)\, \mathrm{d}x
		+\int_{\Omega}\frac{\lambda }{r(x)}| u|^{r(x)}\, \mathrm{d}x
		-\int_{\Omega}\frac{1}{h(x)}| u|^{h(x)}\, \mathrm{d}x 
				-\beta\int_{\Omega}F(x,u)\, \mathrm{d}x.
	\end{equation}
By a standard calculation, one can see that $E_{\lambda,\beta} \in  C^{1}(X,\mathbb{R})$ and the Fréchet derivative is
\begin{equation}\label{Fréchet}	
	\langle E_{\lambda,\beta}'(u),v\rangle =
	\int_{\Omega} \sum_{i=1}^N
	a_i(x,\partial_{x_i}u)\partial_{x_i}v\, \mathrm{d}x
	+\lambda \int_{\Omega} | u|^{r(x)-2} uv\, \mathrm{d}x
 	\ -\int_{\Omega} | u|^{h(x)-2} uv\, \mathrm{d}x -\beta \int_{\Omega} f(x,u)v\, \mathrm{d}x
\end{equation}
for all $u, v\in X$. Hence, the weak solutions of \eqref{e1.1} coincide with the critical points of $E_{\lambda,\beta}$.

To prove Theorem \ref{thm1}, we shall apply the Mountain Pass Theorem \ref{PMT}. We shall begin with the following lemmas.
\begin{lemma} \label{lembounded}
	If assumptions {\rm ($\textbf{\textit{A}}_2$)} and {\rm ($\textbf{\textit{f}}_1$)-($\textbf{\textit{f}}_2$)} are satisfied, and $\left\lbrace u_n\right\rbrace_{n\in \mathbb{N}} $ is a (PS)-sequence for the functional $E_{\lambda, \beta }$, then for all $\lambda\in\mathbb{R}$, the sequence $\left\lbrace u_n\right\rbrace{n\in \mathbb{N}} $ is bounded.
\end{lemma}
\begin{proof} \label{ee4.1}
	Let $\{u_{n}\}_{n\in \mathbb{N}}\subset X$ be a (PS)-sequence. This implies that
	\begin{equation}\label{PsCv}
		E_{\lambda, \beta}(u_{n})=C_{\lambda, \beta } +o_{n}(1)\   \text{ and } \
		\langle E_{\lambda,\beta}'(u_n),v\rangle= o_{n}(1) \text{ for all } v \in X.
	\end{equation}
	Now, by using  {\rm ($\textbf{\textit{f}}_2$)},  we get for a sufficiently large $n$
	\begin{align*}
		C_{\lambda, \beta } +o_{n}(1)
		&\geq E_{\lambda, \beta}(u_{n})-   \frac{1}{\theta_{\lambda}}\langle E_{\lambda, \beta}'(u_{n}),u_{n}\rangle \\
		&\geq
		\int_{\Omega}\sum_{i=1}^N
		A_i(x,\partial_{x_i}u_n)\,dx
		-\frac{1}{\theta_{\lambda}}\int_{\Omega} \sum_{i=1}^N
		a_i(x,\partial_{x_i}u_n)\partial_{x_i}u_n\, \mathrm{d}x \\
		&\quad 	+ \lambda\Big(\frac{1}{\tilde{r}}-\frac{1}{\theta_{\lambda}}\Big)
		\int_{\Omega}| u_{n}|^{r(x)}\, \mathrm{d}x +\int_{\Omega}\Big(\frac{1}{\theta_{\lambda}}-\frac{1}{h^{-}}
		\Big)| u_{n}|^{h(x)}\, \mathrm{d}x\\
		&\quad
		+ \beta \int_{\Omega}\left(F(x,u_{n})- f(x,u_{n})\frac{u_n}{\theta_{\lambda}}\right)\, \mathrm{d}x\\
		&\geq \sum_{i=1}^N\int_{\Omega} \Big[
		A_i(x,\partial_{x_i}u_n)
		-\frac{1}{\theta_{\lambda}}
		a_i(x,\partial_{x_i}u_n)\partial_{x_i}u_n\,\Big] \mathrm{d}x.
	\end{align*}
	where $ \tilde{r}:\equiv r^{+}$ if  $\lambda> 0$ and $\tilde{r}:\equiv r^{-}$ if $\lambda\leq 0$.
	By using assumption $(\textbf{\textit{A}}_2)$, for all $x\in \Omega$ and $i\in \{1,\dots,N\}$ we obtain
	$a_i(x,\partial_{x_i}u_n)\partial_{x_i}u_n \leq p_i(x)A_i(x,\partial_{x_i}u_n) \leq p_M^{+}A_i(x,\partial_{x_i}u_n), $
	which implies
	$ -\frac{1}{\theta_{\lambda}}a_i(x,\partial_{x_i}u_n)\partial_{x_i}u_n  \geq -\frac{p_M^{+}}{\theta_{\lambda}}A_i(x,\partial_{x_i}u_n).
	$
	
	Therefore,
	\begin{align} \label{PS6.4}
		C_{\lambda, \beta } +o_{n}(1)
		&\geq \Big( 1-\frac{p_M^{+}}{\theta_{\lambda}}\Big)
		\sum_{i=1}^N\int_{\Omega}
		A_i(x,\partial_{x_i}u_n) \mathrm{d}x.
	\end{align} 	
	By using again assumption $(\textbf{\textit{A}}_2)$, we have for all $x\in \Omega$ and all $i\in \{1,\dots,N\}$, that
	$ 	A_i(x,\partial_{x_i}u_n) \geq \frac{1}{p_i(x)}|\partial_{x_i}u_n|^{p_i(x)}\geq \frac{1}{p_M^{+}}|\partial_{x_i}u_n|^{p_i(x)}.
	$
	Then we get
	\begin{align} \label{PS6.4}
		C_{\lambda, \beta } +o_{n}(1)
		&\geq \Big( \frac{1}{p_M^{+}}-\frac{1}{\theta_{\lambda}}\Big)
		\sum_{i=1}^N\int_{\Omega}
		|\partial_{x_i}u_n|^{p_i(x)} \mathrm{d}x.
	\end{align} 	
	For any $n \in \mathbb{N}$, we denote by $\mathcal{B}_{n_1}$ and $\mathcal{B}_{n_2}$ the indices sets
	\[\mathcal{B}_{n_1} =\{ i \in \{1,2,\dots,N\} : |\partial_{x_i}u_n|_{p_i(x)}\leq 1\},
	\]
	and
	\[
	\mathcal{B}_{n_2} =\{ i \in \{1,2,\dots,N\} : |\partial_{x_i}u_n|_{p_i(x)}> 1\}.\]
	Applying relations \eqref{L24}, \eqref{L25} and inequality \eqref{PS6.4}, we find
	\begin{align*}
		C_{\lambda, \beta }+o_{n}(1)
		&\geq\Big(\frac{1}{p_{M}^{+}}-\frac{1}{\theta_{\lambda}}\Big)\Big(\sum_{i\in \mathcal{B}_{n_1}}
		\big\|\partial_{x_i}u_n\big\|_{L^{p_i(x)}(\Omega)}^{p_{M}^{+}}
		+ \sum_{i\in \mathcal{B}_{n_2}}
		\big\|\partial_{x_i}u_n\big\|_{L^{p_i(x)}(\Omega)}^{p_{m}^{-}}
		\Big) \\
		& =\Big(\frac{1}{p_{M}^{+}}-\frac{1}{\theta_{\lambda}}\Big)\Bigg[
		\sum_{i=1}^{N} \big\|\partial_{x_i}u_n\big\|_{L^{p_i(x)}(\Omega)}^{p_{m}^{-}} \\
		&\qquad
		- \sum_{i\in \mathcal{B}_{n_1}} \Big(
		\big\|\partial_{x_i}u_n\big\|_{L^{p_i(x)}(\Omega)}^{p_{m}^{-}}
		-\big\|\partial_{x_i}u_n\big\|_{L^{p_i(x)}(\Omega)}^{p_{M}^{+}}
		\Big) \Bigg]\\
		& \leq \Big(\frac{1}{p_{M}^{+}}-\frac{1}{\theta_{\lambda}}\Big)\Big(
		\sum_{i=1}^{N} \big\|\partial_{x_i}u_n\big\|_{L^{p_i(x)}(\Omega)}^{p_{m}^{-}}
		-N\Big).
	\end{align*}
	By relation \eqref{Jensen}, we get
	$$
	C_{\lambda, \beta }+o_{n}(1)
	\geq\Big(\frac{1}{p^{+}_{M}}-\frac{1}{\theta_{\lambda}}\Big)
	\Bigg( \frac{\big\|\partial_{x_i}u_n\big\|_{\overrightarrow{p}(x)}^{p_{m}^{-}}}{N^{p_{m}^{-}-1}}-N
	\Bigg) .
	$$
	Hence, $\{u_{n}\}_{n\in \mathbb{N}}$ is bounded in $X$.  This completes the proof Lemma \ref{lembounded}. 
\end{proof}
\begin{lemma} \label{lemconv}
	Let $\{u_{n}\}_{n\in \mathbb{N}} \subset X$  be a (PS)-sequence with energy level $C_{\lambda, \beta }$. If
	\begin{equation*}
		\begin{aligned}
		C_{\lambda, \beta } < \Big(\frac{1}{\theta_{\lambda}}-\frac{1}{h_{\mathcal{A}}^{-}}
			\Big) \min \Bigg\{\inf_{j\in J}\bigg((\min_{1\leq i\leq N}k_i)^\frac{1}{p_{M}^{+}} N^{1-p_M^+}S_h\bigg)^{\frac{p_{m}^{\ast}(x_j)p_{M}^{+}}{p_{m}^{\ast}(x_j)-p_{M}^{+}}}, \\
			\inf_{j\in J}\bigg( (\min_{1\leq i\leq N}k_i)^\frac{1}{p_{m}^{-}}N^{1-p_M^+}S_h\bigg)^{\frac{p_{m}^{\ast}(x_j)p_{m}^{-}}{p_{m}^{\ast}(x_j)-p_{m}^{-}}}\Bigg\},
	\end{aligned}
	\end{equation*}
	where $S_h $ is defined in relation \eqref{24}, $h_{\mathcal{A}}^{-} :=\inf_{x\in \mathcal{A}}h(x)$ and the sets $\mathcal{A}$ and $J$ are given in Theorem \ref{ccp}, then there exists a strongly convergent subsequence in $X$.
\end{lemma}
\begin{proof}	
	We can divide the proof into two claims.
	
	\textbf{Claim 1.} $u_n\to u$ strongly in $L^{h(x)}(\Omega)$ as $n \to \infty$.
	
	By applying Lemma \ref{lembounded}, we know that $\{u_{n}\}_{n\in \mathbb{N}}$ is bounded in $X$. We can pass to a subsequence, still labeled $\{u_{n}\}_{n\in \mathbb{N}}$, which converges weakly in $X$. Consequently, there exist positive bounded measures $\mu_i, \nu \in \Omega$ such that	
	\begin{equation} \label{convmeasures}
		\sum_{i=1}^{N}|\partial_{x_i} u_n|^{p_i(x)}\rightharpoonup \mu =\sum_{i=1}^{N}\mu_i  \   \text{ and } \
		|u_n|^{h(x)}	\rightharpoonup  \nu .
	\end{equation}	
	Hence, according to Theorem \ref{ccp}, if $J=\emptyset$, then $u_n \to u $ in $L^{h(x)}(\Omega)$. Let us show that if	
	\begin{equation*}
	\begin{aligned}
		C_{\lambda, \beta } < \Big(\frac{1}{\theta_{\lambda}}-\frac{1}{h_{\mathcal{A}}^{-}}
		\Big) \min \Bigg\{\inf_{j\in J}\bigg((\min_{1\leq i\leq N}k_i)^\frac{1}{p_{M}^{+}} N^{1-p_M^+}S_h\bigg)^{\frac{p_{m}^{\ast}(x_j)p_{M}^{+}}{p_{m}^{\ast}(x_j)-p_{M}^{+}}}, \\
		\inf_{j\in J}\bigg( (\min_{1\leq i\leq N}k_i)^\frac{1}{p_{m}^{-}}N^{1-p_M^+}S_h\bigg)^{\frac{p_{m}^{\ast}(x_j)p_{m}^{-}}{p_{m}^{\ast}(x_j)-p_{m}^{-}}}\Bigg\},
	\end{aligned}
\end{equation*}
	and $\{u_{n}\}_{n\in \mathbb{N}}$ is a (PS)-sequence with energy level $C_{\lambda, \beta }$, then $J=\emptyset$. In fact, we assume that $J\neq \emptyset$, and let $x_j\in \mathcal{A}$ be a singular point of the measures $\mu_i$ and $\nu$.
	
	We consider  $\psi \in C^{\infty}_0(\mathbb{R}^N)$, such that
	$0\leq \psi(x) \leq 1$,  $\psi(0)=1$, $\text{supp}\psi\subset B(0,1)$ and $\|\nabla \psi\|_\infty\leq 2$.
	For any $j\in J$ and $\varepsilon>0$, we define the function		
	$	 \psi_{j,\varepsilon} := \psi \Big(\frac{x-x_j}{\varepsilon}\Big)$, for all $\ x \in  \mathbb{R}^N.$		
	Notice that $\psi_{j,\varepsilon} \in C^{\infty}_0(\mathbb{R}^N, \left[ 0,1\right] )$, $\|\nabla \psi_{j,\varepsilon}\|_\infty\leq \frac{2}{\varepsilon}$ and
	\begin{eqnarray*}
		\label{ineq1}
		\psi_{j,\varepsilon}(x)=
		\begin{cases}
			1,  &  x\in  B_{}(x_j,\varepsilon),\\
			
			0, 	&  x\in \mathbb{R}^N\setminus B_{}(x_j,2\varepsilon). \\
		\end{cases}
	\end{eqnarray*}
	Due to the boundedness of $\{u_{n}\}_{n\in \mathbb{N}}$ in $X$, the sequence
	$\{\psi_{j,\varepsilon}u_{n}\}_{n\in \mathbb{N}}$ is also bounded in $X$. Therefore, we have
	$ \langle E_{\lambda, \beta }'(u_n)(\psi_{j,\varepsilon}u_n)\rangle= o_{n}(1),$
	In other words,
	\begin{equation} \label{ee4.7}
		\begin{aligned}
			\sum^{N}_{i=1}\int_{\Omega}\psi_{j,\varepsilon}a_i(x,\partial_{x_i}u_n)\partial_{x_i}u_n\,\mathrm{d}x +\lambda \int_{\Omega} | u|^{r(x)}\psi_{j,\varepsilon}\, \mathrm{d}x
			\\=
			-\sum^{N}_{i=1}	\int_{\Omega}u_na_i(x,\partial_{x_i}u_n)\partial_{x_i}\psi_{j,\varepsilon}\, \mathrm{d}x
			+\int_{\Omega} | u_n|^{h(x)}\psi_{j,\varepsilon}\, \mathrm{d}x 	
			+\beta \int_{\Omega} f(x,u_n)\psi_{j,\varepsilon}u_n\, \mathrm{d}x	+o_{n}(1).
		\end{aligned}
	\end{equation}
   Now, we proceed to prove the following
	\begin{equation}\label{s48}
		\displaystyle  \lim_{\varepsilon\to 0}\Bigg[
		\limsup_{n\to \infty }\Big|\sum^{N}_{i=1}	\int_{\Omega}u_na_i(x,\partial_{x_i}u_n)\partial_{x_i}\psi_{j,\varepsilon}\, \mathrm{d}x
		\Big|
		\Bigg]=0.
	\end{equation}
	It is worth noting that, thanks to the hypotheses $(\textbf{\textit{A}}_1)$, we only need to establish	
	\begin{equation}\label{s481}
		\displaystyle  \lim_{\varepsilon\to 0}\Bigg[
		\limsup_{n\to \infty }\Big|\sum^{N}_{i=1}	\int_{\Omega}c_{a_i} g_i(x)u_n
		\partial_{x_i}\psi_{j,\varepsilon}\, \mathrm{d}x
		\Big|
		\Bigg]=0,
	\end{equation}
	and
	\begin{equation}\label{s482}
		\displaystyle  \lim_{\varepsilon\to 0}\Bigg[
		\limsup_{n\to \infty }\Big|\sum^{N}_{i=1}	\int_{\Omega}c_{a_i}u_{n}\big|\partial_{x_i}
		u_{n}\big|^{p_{i}(x)-1}
		\partial x_{i}\psi_{j,\varepsilon}\, \mathrm{d}x
		\Big|
		\Bigg]=0.
	\end{equation}
	First, we apply the H\"{o}lder inequality and the boundedness of $\{u_n\}_{n\in \mathbb{N}}$ in $X$, to obtain
	\begin{align*}
		&\Big|\int_{\Omega}u_{n}\big|\partial_{x_i}
		u_{n}\big|^{p_{i}(x)-2}\partial x_{i}u_{n}
		\partial x_{i}\psi_{j,\varepsilon}\, \mathrm{d}x
		\Big|
		\leq 	\int_{\Omega}\Big|\partial_{x_i}
		u_{n}\Big|^{p_{i}(x)-1}   \Big|u_n \partial x_{i}\psi_{j,\varepsilon}\Big|\, \mathrm{d}x\\
		&	\leq 2
		\left\|  \left|	\partial x_{i} u_{n}\right|^{p_i(x)-1}\right\| _{L^{\frac{p_i(x)}{p_i(x)-1}}(\Omega)}\left\| 	\partial x_{i}\psi_{j,\varepsilon} u_{n}\right\| _{L^{p_i(x)}(\Omega)}\\
		&	\leq C \max\Big\{ \Big( 	\int_{\Omega}|u_n|^{p_i(x)}| \partial x_{i}\psi_{j,\varepsilon}|^{p_i(x)}\Big)^{\frac{1}{p_i^{-}}},
		\Big( 	\int_{\Omega}|u_n|^{p_i(x)}| \partial x_{i}\psi_{j,\varepsilon}|^{p_i(x)}\Big)^{\frac{1}{p_i^{+}}}
		\Big\}.	
	\end{align*}
	Now, by applying Lebesgue's Dominated Convergence Theorem, we have
	\begin{align*}
		&\displaystyle	\Big|\int_{\Omega}u_{n}\big|\partial_{x_i}
		u_{n}\big|^{p_{i}(x)-2}\partial x_{i}u_{n}
		\partial x_{i}\psi_{j,\varepsilon}\, \mathrm{d}x
		\Big|\\
		&		\leq C \max\Big\{ \Big( 	\int_{\Omega}|u|^{p_i(x)}| \partial x_{i}\psi_{j,\varepsilon}|^{p_i(x)}\Big)^{\frac{1}{p_i^{-}}},
		\Big( 	\int_{\Omega}|u|^{p_i(x)}| \partial x_{i}\psi_{j,\varepsilon}|^{p_i(x)}\Big)^{\frac{1}{p_i^{+}}}
		\Big\}.	
	\end{align*}
	Furthermore, by the H\"{o}lder inequality
	$$ 	\int_{\Omega}|u|^{p_i(x)}| \partial x_{i}\psi_{j,\varepsilon}|^{p_i(x)} \, \mathrm{d}x\\
	 \leq
		C \Big\| |u|^{p_i(x)} \Big\|_{L^{\frac{N}{N-p_i(x)}}(B(x_j,2\varepsilon))}
		\Big\| |\partial x_{i}\psi_{j,\varepsilon}|^{p_i(x)} \Big\|_{L^{\frac{N}{p_i(x)}}(B(x_j,2\varepsilon))}.
	$$
	From
	$ \displaystyle\int_{B(x_j,2\varepsilon)}|\partial x_{i}\psi_{j,\varepsilon}|^{N}\, \mathrm{d}x
	=  \int_{B(0,2)}|\partial x_{i}\psi_{j,\varepsilon}|^{N}\, \mathrm{d}x,$
	we derive
	\begin{align*}
		&	\displaystyle\Big\| |\partial x_{i}\psi_{j,\varepsilon}|^{p_i(x)} \Big\|_{L^{\frac{N}{p_i(x)}}(B(x_j,2\varepsilon))}\\
		& \leq   \max\Bigg\{ \Big( 	\int_{B(x_j,2\varepsilon)}|\partial x_{i}\psi_{j,\varepsilon}|^{N}\mathrm{d}x\Big)^{\frac{1}{\big(\frac{N}{p_i(x)}\big)^{+}}}\\
		&\qquad
\times
		\Big( 	\int_{B(x_j,2\varepsilon)}|\partial x_{i}\psi_{j,\varepsilon}|^{N}\mathrm{d}x\Big)^{\frac{1}{\big(\frac{N}{p_i(x)}\big)^{-}}}
		\Bigg\} \leq C,
	\end{align*}
	for some positive constant $C$, which is independent of $\varepsilon$. Therefore,
	\begin{align*}
	&	\displaystyle \limsup_{n\to \infty }	\Big|\int_{\Omega}u_{n}\big|\partial_{x_i}
		u_{n}\big|^{p_{i}(x)-2}\partial x_{i}u_{n}
		\partial x_{i}\psi_{\varepsilon}\, \mathrm{d}x
		\Big|\\
		& \leq  C \max \Bigg\{\Big\| |u|^{p_i(x)} \Big\|_{L^{\frac{N}{N-p_i(x)}}(B(x_j,2\varepsilon))}^{\frac{1}{p_i^{-}}},\Big\| |u|^{p_i(x)} \Big\|_{L^{\frac{N}{N-p_i(x)}}(B(x_j,2\varepsilon))}^{\frac{1}{p_i^{+}}}
		\Bigg\}.
	\end{align*}
	However,
	$$ \Big\| |u|^{p_i(x)} \Big\|_{L^{\frac{N}{N-p_i(x)}}(B(x_j,2\varepsilon))}\\
		 \leq   \max\Bigg\{ \Big( 	\int_{B(x_j,2\varepsilon)}|u|^{p_i^{\ast}(x)}\mathrm{d}x\Big)^{\frac{1}{\big(\frac{N}{N-p_i(x)}\big)^{+}}},
		\Big( 	\int_{B(x_j,2\varepsilon)}|u|^{p_i^{\ast}(x)}\mathrm{d}x\Big)^{\frac{1}{\big(\frac{N}{N-p_i(x)}\big)^{-}}}
		\Bigg\},
	$$
	so it follows that
	$$ \lim_{\varepsilon\to 0} \limsup_{n\to \infty }	\Big|\int_{\Omega}u_{n}\big|\partial_{x_i}
	u_{n}\big|^{p_{i}(x)-1}
	\partial x_{i}\psi_{j,\varepsilon}\, \mathrm{d}x
	\Big|
	=0.
	$$
	Similarly, we can check \eqref{s481}. Hence, we have completed the proof of  \eqref{s48}.\par
	
	On the other hand, by using  assumption $(\textbf{\textit{f}}_1)$, Theorem \ref{theo23}, and Lebesgue's Dominated Convergence Theorem, we see that
	\begin{equation} \label{ee49}			
		\lim_{n\to\infty} \int_{\Omega}f(x,u_n)u_n\psi_{j,\varepsilon}\, \mathrm{d}x=
		\int_{\Omega}f(x,u)u\psi_{j,\varepsilon}\, \mathrm{d}x, 
		\ 
		\lim_{n\to\infty}\int_{\Omega}| u_n|^{r(x)}\psi_{j,\varepsilon}\, \mathrm{d}x
		=\int_{\Omega}| u|^{r(x)}\psi_{j,\varepsilon}\, \mathrm{d}x.
			\end{equation}
	Thus, when $\varepsilon\to 0$, we get
	\begin{equation} \label{ee50}
		\lim_{\varepsilon\to0}
		\int_{\Omega}f(x,u)u\psi_{\varepsilon}\, \mathrm{d}x=0
		\  \text{ and }
		\lim_{\varepsilon\to0}\int_{\Omega}| u|^{r(x)}\psi_{j,\varepsilon}\, \mathrm{d}x=0.
	\end{equation}
	On the other hand,
	$$
	\lim_{\varepsilon \to 0}\int_{\Omega}\psi_{j,\varepsilon}d\mu_{i}
	=\mu_i \psi(0) \   \text{and }
	\lim_{\varepsilon \to 0}\int_{\partial\Omega}\psi_{j,\varepsilon}d\nu
	=\nu \psi(0).
	$$
	Since $\psi_{j,\varepsilon}$ has compact support, by  taking  the limit $n\to \infty$ and $\varepsilon \to 0$ in \eqref{ee4.7}, from relations \eqref{s48}, \eqref{ee49} and \eqref{ee50}, we get
	$$
	\lim_{\varepsilon \to 0} \left[
	\limsup_{n\to +\infty } \Big(\sum^{N}_{i=1}\int_{\Omega}\psi_{j,\varepsilon}a_i(x,\partial_{x_i}u_n)\partial_{x_i}u_n\,\mathrm{d}x\Big)\right] =\nu_{j}.	
	$$
	By assumption $(\textbf{\textit{A}}_2)$, we have
	\begin{equation}
		\lim_{\varepsilon \to 0} \left[
		\limsup_{n\to +\infty } \Big(\sum^{N}_{i=1}\int_{\Omega}k_i\psi_{j,\varepsilon}|\partial_{x_i}u_n|^{p_i(x)}\,\mathrm{d}x\Big)\right]\leq \nu_{j},	
	\end{equation}
	hence
	\begin{equation}
		\min_{ i\in \{1,\dots,N\}}k_i \mu_{j} \leq \nu_{j}  \text{ for all  } j \in J.
	\end{equation}
Therefore, by invoking relation \eqref{23}, we can deduce that
		\begin{equation*}
		\begin{aligned}
		\nu_j \geq \min &\Bigg\{\bigg((\min_{1\leq i\leq N}k_i)^\frac{1}{p_{M}^{+}} N^{1-p_M^+}S_h\bigg)^{\frac{p_{m}^{\ast}(x_j)p_{M}^{+}}{p_{m}^{\ast}(x_j)-p_{M}^{+}}}, 
		 \bigg( (\min_{1\leq i\leq N}k_i)^\frac{1}{p_{m}^{-}}N^{1-p_M^+}S_h\bigg)^{\frac{p_{m}^{\ast}(x_j)p_{m}^{-}}{p_{m}^{\ast}(x_j)-p_{m}^{-}}}\Bigg\}   \text{ for all  } j \in J.
	\end{aligned}
  \end{equation*}
	Consequently,
	\begin{align*}
		\int_{\Omega}|u_n|^{h(x)}\mathrm{d}x &\to \int_{\Omega}\mathrm{d}\nu 
		\geq \int_{\Omega}|u|^{h(x)}\mathrm{d}x
			+ \min \Bigg\{\bigg((\min_{1\leq i\leq N}k_i)^\frac{1}{p_{M}^{+}} N^{1-p_M^+}S_h\bigg)^{\frac{p_{m}^{\ast}(x_j)p_{M}^{+}}{p_{m}^{\ast}(x_j)-p_{M}^{+}}}, \\
	&\qquad	\bigg( (\min_{1\leq i\leq N}k_i)^\frac{1}{p_{m}^{-}}N^{1-p_M^+}S_h\bigg)^{\frac{p_{m}^{\ast}(x_j)p_{m}^{-}}{p_{m}^{\ast}(x_j)-p_{m}^{-}}}\Bigg\}\sum_{j\in J}\delta_{x_j}\\
		&\geq \int_{\Omega}|u|^{h(x)}\mathrm{d}x 	+\min \Bigg\{\bigg((\min_{1\leq i\leq N}k_i)^\frac{1}{p_{M}^{+}} N^{1-p_M^+}S_h\bigg)^{\frac{p_{m}^{\ast}(x_j)p_{M}^{+}}{p_{m}^{\ast}(x_j)-p_{M}^{+}}},  \\
		&\qquad
		\bigg( (\min_{1\leq i\leq N}k_i)^\frac{1}{p_{m}^{-}}N^{1-p_M^+}S_h\bigg)^{\frac{p_{m}^{\ast}(x_j)p_{m}^{-}}{p_{m}^{\ast}(x_j)-p_{m}^{-}}}\Bigg\}\text{Card}J.
	\end{align*}
	If Card$(J)=\infty$, we get a contradiction.
	On the other hand, by assumptions $(\textbf{\textit{A}}_2)$  and $(\textbf{\textit{f}}_2)$, we have
	\begin{align*}
	&	E_{\lambda, \beta }(u_{n})-   \frac{1}{\theta_{\lambda}}\langle E_{\lambda, \beta}'(u_{n}),u_{n}\rangle	
		\geq \Big(\frac{1}{p^{+}_{M}}-\frac{1}{\theta_{\lambda}}\Big)\sum^{N}_{i=1}
		\int_{\Omega}
		\big|\partial_{x_i}u_n\big|^{p_{i}(x)}
		\, \mathrm{d}x   +\lambda\Big(\frac{1}{\tilde{r}}-\frac{1}{\theta_{\lambda}}\Big)
		\int_{\Omega}| u_{n}|^{r(x)}\, \mathrm{d}x\\
		&\    +\int_{\Omega}\Big(\frac{1}{\theta_{\lambda}}-\frac{1}{h^{-}}
		\Big)| u_{n}|^{h(x)}\, \mathrm{d}x +	\beta  \int_{\Omega}\left(F(x,u_{n})- f(x,u_{n})\frac{u_n}{\theta_{\lambda}}\right)\, \mathrm{d}x
		\geq \Big(\frac{1}{\theta_{\lambda}}-\frac{1}{h^{-}}
		\Big)   \int_{\Omega}| u_{n}|^{h(x)}\, \mathrm{d}x .
	\end{align*}
	Now, setting $\mathcal{A}_{\tau} =\cup_{x\in \mathcal{A}}(B(x,\tau)\cap \Omega)= \{ x\in \Omega : \text{dist}(x, \mathcal{A})<\tau \}$, when $n\to +\infty$, we find
	\begin{align*}
		C_{\lambda,\beta}
		&\geq  \Big(\frac{1}{\theta_{\lambda}}-\frac{1}{h_{\mathcal{A}_{\tau}}^{-}}
		\Big)  \Big( \int_{\Omega}| u|^{h(x)}\, \mathrm{d}x + \sum_{j\in J} \nu_j\delta_{x_j}\Big)\\
		&\geq 	\Big(\frac{1}{\theta_{\lambda}}-\frac{1}{h_{\mathcal{A}_{\tau}}^{-}}
		\Big)\min \Bigg\{\inf_{j\in J}\bigg((\min_{1\leq i\leq N}k_i)^\frac{1}{p_{M}^{+}} N^{1-p_M^+}S_h\bigg)^{\frac{p_{m}^{\ast}(x_j)p_{M}^{+}}{p_{m}^{\ast}(x_j)-p_{M}^{+}}}, \\
		& \qquad \inf_{j\in J}\bigg( (\min_{1\leq i\leq N}k_i)^\frac{1}{p_{m}^{-}}N^{1-p_M^+}S_h\bigg)^{\frac{p_{m}^{\ast}(x_j)p_{m}^{-}}{p_{m}^{\ast}(x_j)-p_{m}^{-}}}\Bigg\}.
	\end{align*}
	Since $\tau>0$ and arbitrary, and $h$ is continuous, we obtain
	\begin{align*}
		C_{\lambda,\beta}
		&\geq 	\Big(\frac{1}{\theta_{\lambda}}-\frac{1}{h_{\mathcal{A}}^{-}}
		\Big) \min \Bigg\{\inf_{j\in J}\bigg((\min_{1\leq i\leq N}k_i)^\frac{1}{p_{M}^{+}} N^{1-p_M^+}S_h\bigg)^{\frac{p_{m}^{\ast}(x_j)p_{M}^{+}}{p_{m}^{\ast}(x_j)-p_{M}^{+}}},\\
	&	\qquad  \inf_{j\in J}\bigg( (\min_{1\leq i\leq N}k_i)^\frac{1}{p_{m}^{-}}N^{1-p_M^+}S_h\bigg)^{\frac{p_{m}^{\ast}(x_j)p_{m}^{-}}{p_{m}^{\ast}(x_j)-p_{m}^{-}}}\Bigg\}.
	\end{align*}
	Therefore, if the condition
	\begin{align*}
		C_{\lambda,\beta}
		&<	\Big(\frac{1}{\theta_{\lambda}}-\frac{1}{h_{\mathcal{A}}^{-}}
		\Big) \min \Bigg\{\inf_{j\in J}\bigg((\min_{1\leq i\leq N}k_i)^\frac{1}{p_{M}^{+}} N^{1-p_M^+}S_h\bigg)^{\frac{p_{m}^{\ast}(x_j)p_{M}^{+}}{p_{m}^{\ast}(x_j)-p_{M}^{+}}}, \\
		&\qquad \inf_{j\in J}\bigg( (\min_{1\leq i\leq N}k_i)^\frac{1}{p_{m}^{-}}N^{1-p_M^+}S_h\bigg)^{\frac{p_{m}^{\ast}(x_j)p_{m}^{-}}{p_{m}^{\ast}(x_j)-p_{m}^{-}}}\Bigg\},
	\end{align*}
   holds, then the index set $J$ is empty, and consequently,
	$\rho_{h}(u_n)\to \rho_{h}(u)$ as $ n\to \infty$. Thus, by applying Lemma \ref{lema 2ccp} and relation \eqref{L26}, we can conclude that $u_n \to u$ strongly in $L^{h(x)}(\Omega)$ as $n \to +\infty$. This completes the proof of Claim 1.\\
	
	\textbf{Claim 2.}  $u_n\to u$ strongly in $X$ as
	$n\to +\infty$.
	
	Since, $\{u_n\}_{n\in \mathbb{N}}$ is bounded in $X$ and $X$ is a reflexive space, there exists a subsequence, still denoted by $\{u_n\}_{n\in \mathbb{N}}$ and
	$u\in X$ such that
	\begin{equation}\label{L15}
		u_n\rightharpoonup u \  \text{weakly in } X.
	\end{equation}
	By Theorem \ref{theo23}, we know that
	$X$  is compactly embedded in
	$L^{h(x)}(\Omega)$, where $1\leq h(x) \leq p_{m}^{\ast}(x)$.
	Therefore, since $u_n\rightharpoonup u$ in the Banach space
	$X$, we can infer that
	\begin{equation}\label{stea}
		u_n\to u\  \text{in  $L^{h(x)}(\Omega)$}.
	\end{equation}
	Using \eqref{PsCv} and \eqref{L15} and the fact that
	$$
	|\langle E_{\lambda, \beta }'(u_n), u_n-u\rangle|\leq\|E_{\lambda, \beta }'(u_n)\|_{X^*}\,
	\|u_n-u\|_{X},
	$$
	we see that
	$
\displaystyle	\lim_{n\to\infty}|\langle E_{\lambda, \beta }'(u_n),u_n-u\rangle|=0,
	$
	that is,
	\begin{equation}\label{star0}
		\begin{aligned}
			&\lim_{n\to\infty}\Bigg[\int_\Omega\sum_{i=1}^N
			a_i(x,\partial_{x_i}u_n)
			\big(\partial_{x_i}u_n-\partial_{x_i}u\big)\,dx+\lambda\int_\Omega
			|u_n|^{r(x)-2}u_n(u_n-u)\,dx\\
			&\   \qquad\   -\int_\Omega
			|u_n|^{h(x)-2}u_n(u_n-u)\,dx
			-\beta \int_\Omega f(x,u_n)(u_n-u)\,dx\Bigg]=0.
		\end{aligned}
	\end{equation}
	Hence, by applying the H\"older inequality, we get
	\begin{equation}\label{saispeprim}
		\begin{aligned}
			\Big|\int_\Omega |u_n|^{r(x)-2}u_n(u_n-u)\,dx\Big|
			&\leq 2\big\||u_n|^{r(x)-1}\big\|_{L^{r'(x)}(\Omega)}
			\big\|u_n-u\big\|_{L^{r(x)}(\Omega)}.
		\end{aligned}
	\end{equation}
	Assume to the contrary, that
	$\big\||u_n|^{r(x)-1}\big\|_{L^{r'(x)}(\Omega)}\to\infty$,  and using relation \eqref{L26}, we obtain
	$$
	\int_\Omega\big(|u_n|^{r(x)-1}\big)^{r'(x)}\,dx\to+\infty
	\text{ if and only if}
	\int_\Omega\big(|u_n|^{r(x)}\big)\,dx\to+\infty,$$
which, in turn, implies $ \|u_n\|_{L^{r(x)}(\Omega)}\to+\infty.
	$
	However, we know that $$\|u_n\|_{L^{r(x)}(\Omega)}\to \|u\|_{L^{r(x)}(\Omega)},$$ so we get a contradiction. Therefore, by relation \eqref{saispeprim} and Theorem \ref{theo23}, we obtain
	\begin{equation}\label{star1}
		\lim_{n\to\infty}\int_\Omega |u_n|^{r(x)-2}u_n(u_n-u)\,dx=0.
	\end{equation}
	Moreover, given assumption {\rm ($\textbf{\textit{f}}_2$)} and the H\"older inequality, we derive
	\begin{align*}
	&	\Big|\int_\Omega f(x,u_n)(u_n-u)\,dx\Big|
		\leq \int_\Omega |f(x,u_n)||u_n-u|\,dx\\
		&\leq
		C_f\int_\Omega|u_n-u|\,dx+C_f\int_\Omega |u_n|^{\ell(x)-1}|u_n-u|\,dx\\
		&\leq C_f\|u_n-u\|_{L^1(\Omega)}+2C_f\||u_n|^{\ell(x)-1}
		\|_{L^{\ell'(x)}(\Omega)}\|u_n-u\|_{L^{\ell(x)}(\Omega)},
	\end{align*}
	so, by relatios \eqref{stea} and \eqref{L26}, we get as above
	\begin{equation}\label{star2}
		\lim_{n\to\infty}\int_\Omega f(x,u_n)(u_n-u)\,dx=0,
	\end{equation}
	hence, by combining relations \eqref{star0}, \eqref{star1} and \eqref{star2}, we find
	\begin{equation}\label{L4155}
		\lim_{n\to\infty}\int_\Omega\sum_{i=1}^N	a_i(x,\partial_{x_i}u_n)
		\big(\partial_{x_i}u_n-\partial_{x_i}u\big)\,dx=0.
	\end{equation}
	By assumption {\rm ($\textbf{\textit{A}}_2$)}, we obtain	
	\begin{equation}\label{L415}
		\lim_{n\to\infty}\int_\Omega\sum_{i=1}^N	|\partial_{x_{i}}u_n|^{p_{i}(x)-2}\partial_{x_{i}}u_n
		\big(\partial_{x_i}u_n-\partial_{x_i}u\big)\,dx=0,
	\end{equation}
	and since $u_n\rightharpoonup u$ in $X$, we have
	\begin{equation}\label{L416}
		\lim_{n\to\infty}\int_\Omega\sum_{i=1}^N	|\partial_{x_{i}}u|^{p_{i}(x)-2}\partial_{x_{i}}u
		\big(\partial_{x_i}u_n-\partial_{x_i}u\big)\,dx = 0.
	\end{equation}
	By combining relations \eqref{L415} and \eqref{L416}, we can infer that
	\begin{equation}\label{L417}
		\lim_{n\to\infty}\sum_{i=1}^N\int_\Omega\Big( 	|\partial_{x_{i}}u_n|^{p_{i}(x)-2}\partial_{x_{i}}u_n
		-|\partial_{x_{i}}u|^{p_{i}(x)-2}\partial_{x_{i}}u\Big)
		\big(\partial_{x_i}u_n-\partial_{x_i}u\big)\,dx=0.
	\end{equation}
	Hence, by employing elementary inequalities (see, e.g., Di Benedetto \cite[Chapter I]{Dib1}), for any $\gamma > 1$, there exists a positive constant $C_{\gamma}$ such that the following inequalities hold
	\begin{eqnarray}
		\label{ineq1}
		\langle |\xi|^{\gamma-2}\xi -|\varrho |^{\gamma-2}\varrho, \xi -\varrho \rangle  &\geq
		\begin{cases}
			C_{\gamma}|\xi -\varrho|^{\gamma}  &\  \text{if }\gamma \geq2\\
		    &\	\\
			C_{\gamma}\frac{|\xi -\varrho|^{2}}{(|\xi| +|\varrho| )^{2-\gamma}}, (\xi,\varrho) \neq (0,0)	&\  \text{if } 1<\gamma <2
		\end{cases}
	\end{eqnarray}
	for all  $ \xi , \varrho \in \mathbb{R} $.
	So, by relation \eqref{L417} and inequalities \eqref{ineq1}, we see that
	\begin{equation}\label{L419}
		\lim_{n\to\infty}\sum_{i=1}^N\int_\Omega	\Big|\partial_{x_{i}}u_n-\partial_{x_{i}}u\Big|^{p_{i}(x)}\,dx=0.
	\end{equation}
	Therefore, we can conclude that $u_n\to u \text{ strongly in } X$. This completes the proof of Lemma \ref{lemconv}.  
\end{proof}
\section{Proofs of Theorems \ref{thm1} and  \ref{thm2}}\label{sec 5}
\textit{Proof of Theorem \ref{thm1}}.
	The proof is a direct consequence of the Mountain Pass Theorem, along with Lemma \ref{lembounded} and Lemma \ref{lemconv}. More precisely, it suffices to verify that $E_{\lambda,\beta}$ has the mountain pass geometry and that $E_{\lambda,\beta}(tu) <0$ for some $t>0$.
	According to assumption  {\rm ($\textbf{\textit{f}}_2$)}, there exist $R>0$ and $\theta_{\lambda}$ such that
	$
	F(x,\xi)\geq C'_{f}| \xi|^{\theta_{\lambda}}, \   \text{for all } x\in \overline{\Omega},| \xi|\geq R.
	$
	So, for $u\in X$ and any $t> 1$, we have
	\begin{align*}
		E_{\lambda,\beta }(tu)
		&\leq\sum_{i=1}^N\int_{\Omega} c_{a_i}\bigg(|g_i(x)||\partial_{x_i}(t u)|
		+\frac{|\partial_{x_i}(t u)|^{p_i(x)}}{p_i(x)}\,\bigg)dx
		+\lambda\int_{\Omega}\frac{1}{r(x)}| tu|^{r(x)}\, \mathrm{d}x
		-\frac{1}{h^{+}}\int_{\Omega}| tu|^{h(x)}\, \mathrm{d}x\\
		&\   -\beta \int_{\{x\in\Omega:|u(x)|>R\}}C_{f}| t u|^{\theta_{\lambda}}\, \mathrm{d}x
		-\beta \text{ meas}(\Omega)\inf\big\{F(x,\xi):x\in\Omega,\,|\xi|\leq R\big\}\\
		&\leq 2t\max_{1\leq i\leq N}c_{a_i}\sum_{i=1}^N
		\|g_i\|_{L^{p'_i(x)}(\Omega)}
		\|\partial_{x_i}u\|_{L^{p_i(x)}(\Omega)}	
		 +t^{p_M^+}\frac{\max_{1\leq i\leq N}c_{a_i}}{p_m^-}
		\sum_{i=1}^N\int_\Omega|\partial_{x_i}u|^{p_i(x)}\,dx\\
		&\  	+\lambda t^{\tilde{r}} \int_{\Omega}\frac{1}{r(x)}| u|^{r(x)}\, \mathrm{d}x
		-\frac{ t^{h^{-}}}{h^{+}}\int_{\Omega}| u|^{h(x)}\, \mathrm{d}x
		  	-C_{f}t^{\theta_{\lambda}}\int_{\Omega}| u|^{\theta_{\lambda}}\, \mathrm{d}x\\
		&\ -\beta \text{ meas}(\Omega)\inf\big\{F(x,\xi):x\in\Omega,\,|\xi|\leq R\big\},
	\end{align*}
	with $\tilde{r}=r^{+}$ if $\lambda> 0$ and $\tilde{r}=r^{-}$
	if $\lambda\leq 0$. So, by assumptions {\rm ($\textbf{\textit{H}}$)} and {\rm ($\textbf{\textit{f}}_2$)}, we have, for all  $\lambda\in\mathbb{R}$, that
	$E_{\lambda,\beta}(tu)\to-\infty\text{ as } t\to +\infty $.				
	
	On the other hand,  based on the assumptions {\rm ($\textbf{\textit{f}}_2$)} and {\rm ($\textbf{\textit{f}}_3$)}, it follows that for any $\varepsilon >0$,  there exists a constant $C(\varepsilon)$ such that
	\begin{equation}
		| F(x,\xi)|\leq \varepsilon| \xi|^{p^{+}_{M}}+C(\varepsilon)| \xi |^{\ell(x)}		\		
		\hbox{for  a.e.}	\
		x\in\Omega	\
		\hbox{ and  all}	\
		\xi\in\mathbb{R}.
	\end{equation}
	Therefore, we find 			
	\begin{equation} \label{e3.1}
		\begin{aligned}
			E_{\lambda,\beta}(u)
			&\geq \frac{\min_{1\leq i\leq N}k_i}{p^{+}_{M}}\sum^{N}_{i=1}\int_{\Omega} \big|\partial_{x_{i}} u \big|^{p_{i}(x)}\,
			\mathrm{d}x+\lambda\int_{\Omega}\frac{1}{r(x)}| u|^{r(x)}\, \mathrm{d}x \\
			&\   \quad 	-\frac{1}{h^{-}}\int_{\Omega}| u|^{h(x)}\, \mathrm{d}x
			-\beta\int_{\Omega}(\varepsilon| u|^{p^{+}_{M}}+C(\varepsilon)| u|^{\ell(x)})
			\, \mathrm{d}x \\
			&\geq\frac{\min_{1\leq i\leq N}k_i}{P^{+}_{M}N^{P^{+}_{M}-1}}
			\|u\|^{P^{+}_{M}}_{\overrightarrow{p}(x)}
			-\frac{|\lambda|}{r^{-}}\int_{\Omega}| u|^{r(x)}\, \mathrm{d}x\\
			&\   \quad	-\frac{1}{h^{-}}\int_{\Omega}| u|^{h(x)}\, \mathrm{d}x
			-\beta \int_{\Omega}\varepsilon| u|^{p^{+}_{M}}\, \mathrm{d}x  - \beta \int_{\Omega}C(\varepsilon)| u|^{\ell(x)}\, \mathrm{d}x.
		\end{aligned}
	\end{equation}		
	Consider 	$0<  \|u\|_{\overrightarrow{p}(x)} <1$. By using relations \eqref{L24}, \eqref{L25} and Theorem \ref{theo23}, we have		
	\begin{equation} \label{E5.3}
		\begin{split}
			E_{\lambda,\beta }(u)
			&\geq\frac{\min_{1\leq i\leq N}k_i}{p^{+}_{M}N^{P^{+}_{M}-1}}\|u\|_{\overrightarrow{p}(x)}^{P^{+}_{M}}
			-\frac{|\lambda| C}{r^{-}}\max\Big\{\|u\|_{\overrightarrow{p}(x)}^{r^{+}},\|u\|_{\overrightarrow{p}(x)}^{r^{-}}\Big\} \\
			&\quad 	-\frac{1}{h^{-}}C'_{1}\|u\|_{\overrightarrow{p}(x)}^{h^{-}}	
			-\beta \varepsilon C'_{2}\|u\|_{\overrightarrow{p}(x)}^{P^{+}_{M}}-
			\beta C(\varepsilon) C'_{3}\|u\|_{\overrightarrow{p}(x)}^{\ell^{-}}\\
			&\geq
			\Big(\frac{\min_{1\leq i\leq N}k_i}{p^{+}_{M}N^{P^{+}_{M}-1}}
			 - \beta \varepsilon C'_{2}\Big)\|u\|_{\overrightarrow{p}(x)}^{P^{+}_{M}} -\frac{|\lambda| C}{r^{-}}\max\Big\{\|u\|_{\overrightarrow{p}(x)}^{r^{+}},\|u\|_{\overrightarrow{p}(x)}^{r^{-}}\Big\} \\
		&\  \quad -\frac{C'_{1}}{h^{-}}\|u\|_{\overrightarrow{p}(x)}^{h^{-}}
			-\beta C(\varepsilon) C'_{3}\|u\|_{\overrightarrow{p}(x)}^{\ell^{-}}.
		\end{split}
	\end{equation}		
	Set $\varepsilon=\frac{\min_{1\leq i\leq N}\{k_i\}}{2\beta C^{'}_{2} p^{+}_{M}N^{p^{+}_{M}-1}}$ and
	$$
	\Phi(t)=\frac{ \min_{1\leq i\leq N}\{k_i\}}{2p^{+}_{M}N^{p^{+}_{M}-1}}t^{p_M^{+}}
	-\frac{|\lambda| C}{r^{-}}\max\Big\{t^{r^{+}},t^{r^{-}}\Big\}
	-\frac{C'_{1}}{h^{-}}t^{h^{-}}
	-\beta C(\varepsilon) C'_{3}t^{\ell^{-}}.
	$$ Since $p^{+}_{M}< \min\{r^{-},\ell^{-} \}<h^{-}$, we see that there
	exists $\rho>0$ such that $\underset{t\geq0}\max \Phi(t)=\Phi(\rho)$. Hence,
	by \eqref{E5.3}, there exists $\rho>0$ such that
	$	E_{\lambda,\beta }(u)\geq \mathcal{R}>0 \   \text{as } \left\| u\right\| =\rho.$	
	This yields the existence of an element $\tilde{u}$ of $X$ such that $ E_{\lambda,\beta }(\tilde{u})<0$. Consequentely, the critical value is
	$$C_{\lambda,\beta } := \inf_{\phi \in \Gamma} \max_{t\in \left[ 0,1\right]  }E_{\lambda,\beta }(\phi(t)),$$
	where
	$\Gamma= \{ \phi \in C\big( \left[ 0,1\right],X\big) :\phi(0)=0, \phi (1)=\tilde{u}\}. $
	This concludes the proof of Theorem \ref{thm1}. \qed 		
	\\
		\vskip 0.1cm		
In the sequel, we shall  prove under some symmetry condition on the function $f$ that \eqref{e1.1} possesses infinitely many nontrivial solutions in the case $a_i(x,\xi):=|\xi|^{p_i(x)-2}\xi$,  for all $i\in\{1,\dots,N\}$.

\eject

\textit{Proof of Theorem \ref{thm2}}.
	We shall use the $\mathbb{Z}_2$-symmetric version of the Mountain Pass Theorem \ref{SPMT}, to get the proof of Theorem \ref{thm2}. By assumption ($\textbf{\textit{f}}_4$), the function $f$ is even, the functional $E_{\lambda,\beta }$ is even, too. It suffices to check the condition ($\mathcal{I}_2 '$). In fact, by using condition {\rm ($\textbf{\textit{f}}_2$)}, we have	
	$
	F(x,\xi)\geq C_{1}| \xi|^{\theta_{\lambda}}-C_{2},\
	\text{for all }(x,\xi)\in\Omega\times\mathbb{R}.
	$
	Then  there exist positive constants $C'_{1},C'_{2}$ and $ C'_{3}$ such that
		\begin{equation*}
		\begin{split}
	E_{\lambda,\beta}(u)&\leq \frac{C'_{1}}{p^{-}_{m}}\sum^{N}_{i=1}
	\int_{\Omega}\big|\partial_{x_i}u\big|^{p_{i}(x)}\,
	\mathrm{d}x+\frac{\lambda}{\tilde{r}}\int_{\Omega}
	| u|^{r(x)}\, \mathrm{d}x
	 \quad
	-\frac{1}{h^{+}}\int_{\Omega}
	| u|^{h(x)}\, \mathrm{d}x
	-C'_{2}\| u\|_{L^{\theta_{\lambda}}(\Omega)}^{\theta_{\lambda}}+ C'_{3},
			\end{split}
		\end{equation*}
	with $\tilde{r}=r^{-}$ if $\lambda> 0$ and $\tilde{r}=r^{+}$ if
	$\lambda\leq 0$. On the other hand, we have the following inequality
	$ \displaystyle	\sum^{N}_{i=1}\big\|\partial_{x_{i}} u \big\|^{p^{+}_{M}}_{L^{p_{i}(x)}
		(\Omega)}\leq C\Big(\sum^{N}_{i=1}\big\|\partial_{x_{i}} u
	\big\|_{L^{p_{i}(x)}(\Omega)}\Big)^{p^{+}_{M}},
	$
	with $C$ a positive constant. So, by using the last ine\-qua\-li\-ty above and Theorem \ref{theo23}, we obtain, in the case when $\lambda > 0$, that
	$$
	E_{\lambda,\beta}(u)\leq\frac{C'}{p^{-}_{m}}\|u\|_{\overrightarrow{p}(x)}^{p^{+}_{M}}
	+\frac{C_{4}}{\tilde{r}}\|u\|_{\overrightarrow{p}(x)}^{r^{+}}
	-\frac{1}{h^{+}}\int_{\Omega}
	| u|^{h(x)}\, \mathrm{d}x
	-C'_{2}\| u\|_{\overrightarrow{p}(x)}^{\theta_{\lambda}}+C'_{3}.
	$$
	Let $u \in X$ be arbitrary but fixed. We put
	$$ \Omega_{<}=\left\lbrace x\in \Omega  : \left| u(x)\right| <1\right\rbrace  \text{ and } \Omega_{\geq}=\Omega  \backslash \Omega_{<}.$$
	Then  we have	
	\begin{align*}
		E_{\lambda,\beta}(u)	
		&\leq\frac{C'}{p^{-}_{m}}\|u\|_{\overrightarrow{p}(x)}^{p^{+}_{M}}
		+\frac{C_{4}}{\tilde{r}}\|u\|_{\overrightarrow{p}(x)}^{r^{+}}
		-\frac{1}{h^{+}}\int_{\Omega}
		| u|^{h(x)}\, \mathrm{d}x
		-C'_{2}\| u\|_{\overrightarrow{p}(x)}^{\theta_{\lambda}}+C'_{3} \\
		&\leq\frac{C'}{p^{-}_{m}}\|u\|_{\overrightarrow{p}(x)}^{p^{+}_{M}}
		+\frac{C_{4}}{\tilde{r}}\|u\|_{\overrightarrow{p}(x)}^{r^{+}}
		-\frac{1}{h^{+}}\int_{\Omega_{\geq}}
		| u|^{h^-}\, \mathrm{d}x
		-C'_{2}\| u\|_{\overrightarrow{p}(x)}^{\theta_{\lambda}} +C'_{3}\\
		&\leq\frac{C'}{p^{-}_{m}}\|u\|_{\overrightarrow{p}(x)}^{p^{+}_{M}}
		+\frac{C_{4}}{\tilde{r}}\|u\|_{\overrightarrow{p}(x)}^{r^{+}}
		-\frac{1}{h^{+}}\int_{\Omega}
		| u|^{h^-}\, \mathrm{d}x
		+\frac{1}{h^{+}}\int_{\Omega_{<}}
		| u|^{h^-}\, \mathrm{d}x  
		 -C'_{2}\| u\|_{\overrightarrow{p}(x)}^{\theta_{\lambda}} +C'_{3}.
	\end{align*}
	Since there exists $C_4'>0$ such that	
	$
	\frac{1}{h^{+}}\int_{\Omega_{<}}
	| u|^{h^-}\, \mathrm{d}x \leq C_4', \   \text{ for all } u \in X,$
	we get,	
	$
	E_{\lambda,\beta }(u)	
	\leq\frac{C'}{p^{-}_{m}}\|u\|_{\overrightarrow{p}(x)}^{p^{+}_{M}}
	+\frac{C_{4}}{\tilde{r}}\|u\|_{\overrightarrow{p}(x)}^{r^{+}}
	-\frac{1}{h^{+}}\int_{\Omega}
	| u|^{h^-}\, \mathrm{d}x
	-C'_{2}\| u\|_{\overrightarrow{p}(x)}^{\theta_{\lambda}}
	+C_5',
	$
	for all $u\in X$. On the other hand,  consider the functional  $\left| .\right|_{h^-}: X\to \mathbb{R}$, which is defined as follows
	$$\left| u \right|_{h^-} = \Big(\int_{\Omega}| u|^{h^-}dx\Big)^{1/h^-}.$$
	This functional defines a norm on $X$. 	Let $X_1$ be a fixed finite-dimensional subspace of $X$. So, $\left| .\right|_{h^-}$ and $\left\| .\right\|_{\overrightarrow{p}(x)}$ are equivalent norms, implying the existence of a positive constant  $C_{5}=C( X_1)$ such that
	$ \left\| u\right\|_{\overrightarrow{p}(x)}^{h^-} \leq C_{5}\left| u\right|_{h^-}^{h^-}, \text{ for all } u \in X_1.$
	Consequently, we have established the existence of a positive constant $C_{6}$ such that	
	$$
	0 \leq E_{\lambda,\beta }(u)	
	\leq\frac{C'}{p^{-}_{m}}\|u\|_{\overrightarrow{p}(x)}^{p^{+}_{M}}
	+\frac{C_{4}}{\tilde{r}}\|u\|_{\overrightarrow{p}(x)}^{r^{+}}
	- \frac{C_{6}}{h^-}\left\| u\right\|_{\overrightarrow{p}(x)}^{h^-}
	-C'_{2}\| u\|_{\overrightarrow{p}(x)}^{\theta_{\lambda}}
	+C_5'.
	$$
	Given that $\theta_{\lambda}>r^+$ and  $h^->p_M^+$, we can conclude that
	$\mathcal{S}_1$ is bounded in $X$.	
	Hence, by invoking Theorem  \ref{SPMT},  it follows that  $E_{\lambda,\beta}$ possesses an unbounded sequence of critical values, which in turn, implies that problem \eqref{e1.1} has infinitely many weak solutions in $X$. The proof of Theorem \ref{thm2} is thus complete. 
\qed

\subsection*{\small \bf Acknowledgments}
{\small The second author expresses her gratitude to the Faculty of Fundamental Science, Industrial University of Ho Chi Minh City, Vietnam, for the opportunity to work in collaboration. The third author was supported by the Slovenian Research and Innovation Agency grants P1-0292, J1-4031, J1-4001, N1-0278, N1-0114, and N1-0083. The authors express their gratitude to the editor and referees for their valuable comments and suggestions.}

\subsection*{\small \bf Conflict of Interest}
{\small The authors declare to have no conflict of interest.}

\end{document}